\documentclass[12pt,english,fleqn,liststotoc,bibtotoc,idxtotoc,BCOR7.5mm,tablecaptionabove]{amsart}
\usepackage[T1]{fontenc}
\usepackage[latin1]{inputenc}
\usepackage{geometry}
\usepackage{color}
\usepackage{amsthm}
\usepackage{amstext}
\usepackage{amssymb}
\usepackage{amsmath}
\usepackage{graphicx}
\usepackage{young}
\usepackage{xypic}
\usepackage{amscd}
\usepackage[section]{placeins}
\usepackage[svgnames]{xcolor}
\usepackage[unicode=true,
bookmarks=true,bookmarksnumbered=true,bookmarksopen=false,
breaklinks=false,pdfborder={0 0 1},backref=false,colorlinks=true]
{hyperref}
\hypersetup{pdftitle={},
	pdfauthor={Nick Early},
	pdfsubject={},
	pdfkeywords={},
	linkcolor=black, citecolor=black, urlcolor=blue, filecolor=blue,  pdfpagelayout=OneColumn, pdfnewwindow=true,  pdfstartview=XYZ, plainpages=false, pdfpagelabels}
\usepackage{breakurl}

\makeatletter

\theoremstyle{plain}
\newtheorem{thm}{\protect\theoremname}
\theoremstyle{plain}

\theoremstyle{plain}

\theoremstyle{remark}

\theoremstyle{plain}
\newtheorem{lem}[thm]{\protect\lemmaname}
\theoremstyle{plain}
\newtheorem{prop}[thm]{\protect\propositionname}
\theoremstyle{remark}

\theoremstyle{remark}
\newtheorem{rem}[thm]{\protect\remarkname}
\theoremstyle{definition}
\newtheorem{defn}[thm]{\protect\definitionname}
\theoremstyle{definition}
\newtheorem{example}[thm]{\protect\examplename}
\theoremstyle{plain}
\newtheorem{cor}[thm]{\protect\corollaryname}
\theoremstyle{plain}

\newcommand{\cellsize}{20}
\newlength{\cellsz} 
\newsavebox{\cell}
\sbox{\cell}{\begin{picture}(\cellsize,\cellsize)
	\put(0,0){\line(1,0){\cellsize}}
	\put(0,0){\line(0,1){\cellsize}}
	\put(\cellsize,0){\line(0,1){\cellsize}}
	\put(0,\cellsize){\line(1,0){\cellsize}}
	\end{picture}}
\newcommand\cellify[1]{\def\thearg{#1}\def\nothing{}%
	\ifx\thearg\nothing
	\vrule width0pt height\cellsz depth0pt\else
	\hbox to 0pt{\usebox{\cell} \hss}\fi%
	\vbox to \cellsz{
		\vss
		\hbox to \cellsz{\hss$#1$\hss}
		\vss}}
\newcommand\tableau[1]{\vtop{\let\\\cr
		\baselineskip -16000pt \lineskiplimit 16000pt \lineskip 0pt
		\ialign{&\cellify{##}\cr#1\crcr}}}
\newcommand{\kellsize}{20}
\newlength{\kellsz} 
\newsavebox{\kell}
\sbox{\kell}{\begin{picture}(\kellsize,\kellsize)
	\put(0,0){\line(1,0){\kellsize}}
	\put(0,0){\line(0,1){\kellsize}}
	\put(\kellsize,0){\line(0,1){\kellsize}}
	\put(0,\kellsize){\line(1,0){\kellsize}}
	\end{picture}}
\newcommand\kellify[1]{\def\thearg{#1}\def\nothing{}%
	\ifx\thearg\nothing
	\vrule width0pt height\kellsz depth0pt\else
	\hbox to 0pt{\usebox{\kell} \hss}\fi%
	\vbox to \kellsz{
		\vss
		\hbox to \kellsz{\hss$#1$\hss}
		\vss}}
\newcommand\ktableau[1]{\vtop{\let\\\cr
		\baselineskip -16000pt \lineskiplimit 16000pt \lineskip 0pt
		\ialign{&\kellify{##}\cr#1\crcr}}}

\newcommand{\sellsize}{10}
\newlength{\sellsz} 
\newsavebox{\sell}
\sbox{\sell}{\begin{picture}(\sellsize,8)
	\put(0,0){\line(1,0){\sellsize}}
	\put(0,0){\line(0,1){\sellsize}}
	\put(\sellsize,0){\line(0,1){\sellsize}}
	\put(0,\sellsize){\line(1,0){\sellsize}}
	\end{picture}}
\newcommand\sellify[1]{\def\thearg{#1}\def\nothing{}%
	\ifx\thearg\nothing
	\vrule width0pt height\sellsz depth0pt\else
	\hbox to 0pt{\usebox{\sell} \hss}\fi%
	\vbox to \sellsz{
		\vss
		\hbox to \sellsz{\hss$#1$\hss}
		\vss}}
\newcommand\stableau[1]{\vtop{\let\\\cr
		\baselineskip -16000pt \lineskiplimit 16000pt \lineskip 0pt
		\ialign{&\sellify{##}\cr#1\crcr}}}

\newcommand{\ssellsize}{5}
\newlength{\ssellsz} 
\newsavebox{\ssell}
\sbox{\ssell}{\begin{picture}(\ssellsize,4)
	\put(0,0){\line(1,0){\ssellsize}}
	\put(0,0){\line(0,1){\ssellsize}}
	\put(\ssellsize,0){\line(0,1){\ssellsize}}
	\put(0,\ssellsize){\line(1,0){\ssellsize}}
	\end{picture}}
\newcommand\ssellify[1]{\def\thearg{#1}\def\nothing{}%
	\ifx\thearg\nothing
	\vrule width0pt height\sellsz depth0pt\else
	\hbox to 0pt{\usebox{\ssell} \hss}\fi%
	\vbox to \ssellsz{
		\vss
		\hbox to \ssellsz{\hss$#1$\hss}
		\vss}}
\newcommand\sstableau[1]{\vtop{\let\\\cr
		\baselineskip -16000pt \lineskiplimit 16000pt \lineskip 0pt
		\ialign{&\ssellify{##}\cr#1\crcr}}}

\usepackage{ifpdf}
\ifpdf

\IfFileExists{lmodern.sty}
{\usepackage{lmodern}}{}

\fi 

\usepackage{multicol}

\makeatother

\providecommand{\claimname}{\inputencoding{latin9}Claim}
\providecommand{\conjecturename}{\inputencoding{latin9}Conjecture}
\providecommand{\corollaryname}{\inputencoding{latin9}Corollary}
\providecommand{\definitionname}{\inputencoding{latin9}Definition}
\providecommand{\examplename}{\inputencoding{latin9}Example}
\providecommand{\lemmaname}{\inputencoding{latin9}Lemma}
\providecommand{\notename}{\inputencoding{latin9}Note}
\providecommand{\propositionname}{\inputencoding{latin9}Proposition}
\providecommand{\questionname}{\inputencoding{latin9}Question}
\providecommand{\remarkname}{\inputencoding{latin9}Remark}
\providecommand{\theoremname}{\inputencoding{latin9}Theorem}
\providecommand{\problemname}{\inputencoding{latin9}Problem}

\newenvironment{nouppercase}{%
	\renewcommand{\uppercasenonmath}[1]{}}{}

\newcommand\twoheaduparrow{\mathrel{\rotatebox{90}{$\twoheaduparrow$}}}
\newcommand\twoheaddownarrow{\mathrel{\rotatebox{270}{$\twoheaddownarrow$}}}

\begin{document}
	\author{Nick Early}
	\thanks{Perimeter Institute for Theoretical Physics \\
		email: \href{mailto:earlnick@gmail.com}{earlnick@gmail.com}}

	\title[Weakly separated collections and matroid subdivisions]{From weakly separated collections to matroid subdivisions}

\begin{nouppercase}
	\maketitle
\end{nouppercase}
	\begin{abstract}
		We study arrangements of slightly skewed tropical hyperplanes, called blades by A. Ocneanu, on the vertices of a hypersimplex $\Delta_{k,n}$, and we investigate the resulting induced polytopal subdivisions.  We show that placing a blade on a vertex $e_J$ induces an $\ell$-split matroid subdivision of $\Delta_{k,n}$, where $\ell$ is the number of cyclic intervals in the $k$-element subset $J$.  We prove that a given collection of $k$-element subsets is weakly separated, in the sense of the work of Leclerc and Zelevinsky on quasicommuting families of quantum minors, if and only if the arrangement of the blade $((1,2,\ldots, n))$ on the corresponding vertices of $\Delta_{k,n}$ induces a matroid (in fact, a positroid) subdivision.  In this way we obtain a compatibility criterion for (planar) multi-splits of a hypersimplex, generalizing the rule known for 2-splits.  We study in an extended example a matroidal arrangement of six blades on the vertices $\Delta_{3,7}$.

	\end{abstract}
	
	\begingroup
	\let\cleardoublepage\relax
	\let\clearpage\relax
	\tableofcontents
	\endgroup

	\section{Introduction}
	In this paper, we introduce a refinement of the notion of a tropical hyperplane arrangement: the \textit{matroidal blade arrangement}.  The prototypical example of a blade, denoted $((1,2,\ldots, n))$, is the polyhedral complex which is obtained by gluing together the $\binom{n}{2}$ simplicial codimension 1 cones with generating edges a cyclic system of roots $e_i-e_{i+1}$, which partition $\mathbb{R}^{n-1}$ into $n$ maximal cells.  More compactly, it is the union of the facets of the normal fan to the fundamental Weyl alcove.  A precise definition is given in Section \ref{sec:Matroid subdivisions and matroidal blade arrangements}; see also Figure \ref{fig:blade3coordinates000}.  In this paper, by a \textit{blade arrangement} we mean a superposition of a number of copies of $((1,2,\ldots, n))$ on the vertices of a given hypersimplex $\Delta_{k,n}$, where $2\le k\le n-2$.
	
	We will see that any blade arrangement induces an (in general not matroidal, possibly trivial) subdivision.  Sometimes this subdivision is matroidal; in fact, we establish an equivalence between weakly separated collections and a particularly well-behaved subclass of blade arrangements, where the maximal cells are, borrowing terminology from physics, \textit{planar} in nature: they are always positroid polytopes with respect to the given cyclic order.  In this work, our use of blades and their arrangements has motivation coming from several areas, including tropical geometry, moduli spaces and matroid subdivisions \cite{ArdilaBilley,BandeltDress,TropicalConvexity,FraserLe,Horn,Lafforgue,OlarteLocalDressians,Schroter,Speyer2008,Speyer2009} and convex geometry and combinatorics (specifically plabic graphs and alcoved polytopes) \cite{GalashinPostnikov,GalashinPostnikovWilliams,KarpmanKodama,AlcovedPolytopes,LamPostnikov}.
		
		It is also motivated by recent work \cite{TropGrassmannianScattering} of Cachazo, Guevara and Mizera and the author, to generalize the Cachazo-He-Yuan (CHY) formulation \cite
		{CHY1,Cachazo Scattering 2 CHY} of Quantum Field Theory from $\mathbb{CP}^1$ to $\mathbb{CP}^{k-1}$, establishing close contact with the tropical Grassmannian, and in particular the positive tropical Grassmannian.  See also \cite{CachazoRojas,TropicalGrassmannianCluster Drummond,SoftTheoremtropGrassmannian}.  An expanded discussion is included in Section \ref{sec: physical motivation}.

	Blades were first defined by Ocneanu in \cite{OcneanuVideo}; see Definition \ref{defn:blade} and Proposition \ref{prop: blades Minkowski sum} for precise statements, and our previous work \cite{EarlyBlades}, where their study was initiated.  In particular, there it was shown that blades satisfy a Minkowski sum decomposition akin to that of polyhedral cones; additionally a filtered \textit{basis} was introduced which has some intriguing enumerative properties involving a conjecturally symmetric unimodal generating function.  The maximal cells of the subdivision induced by a blade coincide with the tangent cones to faces of the permutohedron \cite{PostnikovPermutohedra}; and they are closely related to the family of polytopes studied by Pitman-Stanley \cite{PitmanStanley}. 
	
	For any $n$-cycle $(\sigma_1,\ldots, \sigma_n)$, the blade $((\sigma_1,\ldots, \sigma_n))$ is isomorphic to the tropical hyperplane defined by $\min\{x_1,\ldots, x_n\}$ (see Proposition \ref{prop: cyclically nonplanar finest refinement}), lifted to some affine hyperplane section where $x_1+\cdots x_n=r$, over the tropical torus $\mathbb{R}^n\slash (1,1,\ldots, 1)\mathbb{R}$, such that all of its edges are ``twisted'' in the direction of roots $e_i-e_j$, rather than in the coordinate directions $e_i$.  This twist will turn out to be exactly what is necessary to guarantee compatibility with constructions involving matroid subdivisions and matroids more generally.  In this paper we study certain restricted arrangements of the blade $((1,2,\ldots, n))$ on the vertices $e_{I_1},\ldots, e_{I_m}$ of hypersimplices 
	$$\Delta_{k,n} = \left\{x\in \lbrack 0,1\rbrack^n : \sum_{i=1}^n x_i=k\right\},$$
	for $k=2,\ldots, n-2$.

	Any such an arrangement induces a (possibly trivial) subdivision of the hypersimplex $\Delta_{k,n}$ into polytopes $\Pi_1,\ldots,\Pi_t$ whose facet inequalities are of the form $x_{i+1}+\cdots +x_j = r_{ij}$; such a polytope is isomorphic to an alcoved polytope, in the sense of \cite{AlcovedPolytopes}.
	
	We shall require all polytopes $\Pi_i$ to be matroid (in particular positroid) polytopes; in this case we call the arrangement of blades \textit{matroidal}.  In fact, in Theorem \ref{thm: weakly separated matroid subdivision blade} we prove that this condition is equivalent to asking that the vertices $e_{I_1},\ldots, e_{I_m}$ define a weakly separated collection of $k$-element subsets $\{I_1,\ldots, I_m\}$.

	The basic example of a blade arrangement that is not matroidal is given in Example \ref{example: octahedral nonmatroidal blade arrangement}, where the octahedron $\Delta_{2,4}$ is fully triangulated by inducing a subdivision from an arrangement of two blades that induce two incompatible 2-splits; these two blades are arranged on the vertices respectively $e_{1}+e_3$ and $e_{2}+e_4$, of $\Delta_{2,4}$.  The key feature which we point out is that, as a pair, the vertices fail to be \textit{weakly separated.}  In the usual geometric interpretation for $k$-element subsets, c.f. \cite{WeakSeparationPostnikov}, two $k$ element subsets $I$ and $J$ are weakly separated if there exists a chord separating the sets $I\setminus J$ and $J\setminus I$ when drawn on a circle.  
	In Section \ref{sec:Permutohedral blades and alcoved polytopes} we recall basic properties of blades from \cite{EarlyBlades}, and show that they are isomorphic to tropical hyperplanes; their explicit presentation as tropical hypersurfaces is given.  Then we introduce planar matroid subdivisions and show that they are refined by the alcove triangulation from \cite{AlcovedPolytopes}.  Section \ref{sec:Matroid subdivisions and matroidal blade arrangements} contains the first main result, Theorem \ref{thm: multisplit blade}.  
	
	This paper originated as an outgrowth of efforts to develop a systematic understanding of the $\mathbb{CP}^{k-1}$ generalization \cite{TropGrassmannianScattering} of the scattering equations \cite{CHY1} building on the existing explicit results for $k=3$ and $n=6,7,8$, see \cite{CachazoRojas,TropicalGrassmannianCluster Drummond,SoftTheoremtropGrassmannian}.
	
	To that end, we point out that Lemma \ref{lem: to the boundary hypersimplex} could be used to compute, for each maximal weakly separated collection of vertices in $\Delta_{3,n}$, a tree arrangement in the sense of \cite{Hermann How to draw}.  We illustrate how this could work in Section \ref{sec: tree arrangements}.  

	\subsection{Physical Motivation}\label{sec: physical motivation}

	Let us explain more carefully how this work came about, as motivated by recent developments in theoretical particle physics.  In this context, the starting point is the Cachazo-He-Yuan (CHY) formalism \cite{CHY1}, which has at its heart the \textit{scattering equations},
	$$\sum_{b\not=a}\frac{s_{a,b}}{\sigma_a - \sigma_b} = 0,$$
	for each $a=1,\ldots, n$, having fixed values of the kinematic data, with \textit{Mandelstam invariants} $(s_{a,b})$, satisfying $s_{a,a}=0$ and $\sum_{a=1}^n s_{a,b}=0$.  These form a system of equations on the configuration space of $n$ punctures $\sigma_1,\ldots, \sigma_n \in \mathbb{CP}^1$ modulo projective equivalence, that is, on the moduli space of $n$-pointed stable curves $\mathcal{M}_{0,n}$.  The scattering equations lead to compact formulas for the calculation of scattering amplitudes for a large number of quantum field theories, including the biadjoint scalar, the nonlinear sigma model, special Galileon, and Yang-Mills \cite{CHYQFTs}.  In \cite{TropGrassmannianScattering}, Cachazo, Guevara, Mizera and the author (CEGM) completed the CHY formalism, with a generalization of the scattering equations, replacing $\mathbb{CP}^1$ with $\mathbb{CP}^{k-1}$ for all $k\ge 2$.  Additionally, CEGM introduced the so-called generalized biadjoint scalar partial amplitudes, certain $\binom{n}{k}-n$ parameter rational functions, denoted $m^{(k)}_n(\alpha,\beta)$ for all $(k,n)$ satisfying $2\le k\le n-2$, for pairs of cyclic orders $\alpha,\beta$ on $\{1,\ldots, n\}$.  The first cases of $m^{(k)}_n(\alpha,\alpha)$ (or for short $m^{(k)}_n$) with $\alpha = (12\cdots n)$ had been evaluated in within the span of about eight months, including for $$(k,n) \in \{(3,6),(3,7),(3,8),(3,9),(4,8),(4,9)\},$$
	using a second formulation of $m^{(k)}_n$ in terms of certain Generalized Feynman Diagrams \cite{CachazoBorges,CGUZ}, as implemented using arrays of metric trees.  See \cite{SongAmplitudes} for a different approach, and more data.  Here $m^{(k)}_n$ is a rational function of homogeneous degree $-(k-1)(n-k-1)$ in a certain collection of $\binom{n}{k}-n$ Mandelstam invariants $s_J$, for $k$-element subsets $J$ of $\{1,\ldots, n\}$.  The systematic evaluation of $m^{(k)}_n$ in general for any $k\ge 3$ is a seemingly fantastically difficult problem due to the inherent complexity of the locus of singularities, which are known to be intimately connected to the rays of the positive tropical Grassmannian $\text{Trop}^+G(k,n)$.  There is a pressing need to develop a systematic approach to calculate $m^{(k)}_n$ for $k\ge 3$ and $n$ large, but it is quite a formidable problem using existing techniques.

	The aim of this work is to introduce some combinatorial techniques that will help investigate the properties of $m^{(k)}_n$.  Let us sketch a brief outline of what such a systematic approach might look like.
	
	By translating the height functions $\mathfrak{h}_{((1,2,\ldots, n))}$ introduced in Section \ref{sec: bladesInTropicalGeometry} around the vertex set of $\Delta_{k,n}$, one constructs the \textit{planar basis} \cite{EarlyPlanarBasis} 
	$$\left\{\eta_{J}(s): J = \{j_1,\ldots, j_k\} \text{ is nonfrozen}\right\}$$
	of linear functions $\eta_J(s)$ on the kinematic space (here $m^{(k)}_n$ is a homogeneous rational function on the kinematic space).  The kinematic space is a codimension n subspace of isomorphic to $\mathbb{R}^{\binom{n}{k}}$; it carries information about the scattering process and that $m^{(k)}_n$ is a function on this space.  The functions $\eta_J$ help to organize the poles of $m^{(k)}_n$ more systematically; in fact, on any hyperplane of the form $\eta_J(s) = 0$ one encounters a pole.  Now $\eta_J$ is generically dual to the height function over the vertices of $\Delta_{k,n}$ which is linear except on the blade $((1,2,\ldots, n))$, translated to the vertex $e_{j_1} + \cdots + e_{j_k}$ of $\Delta_{k,n}$.  For example, one of the coarsest positroidal subdivisions of $\Delta_{3,6}$, having exactly three maximal cells, is induced by the translation of the height function $h_{((1,2,3,4,5,6))}$ to the vertex $e_2+e_4+e_6$ of $\Delta_{3,6}$, so that the height of a given vertex $e_i+e_j+e_k$ maps onto the coefficient of $s_{ijk}$ in 
	$$\eta_{246}(s) = \frac{1}{6} \left(6 s_{123}+5 s_{124}+4 s_{125}+\cdots + s_{245}+5 s_{256}+6 s_{345}+5 s_{346}+4 s_{356}+3 s_{456}\right),$$
	and one can check that modulo the relations induced by lineality (known in physics as generalized momentum conservation),
	$$\sum_{\{i,j,k\}\ni a} s_{ijk} = 0$$
	for each $a=1,\ldots, n$, then one has the more familiar expression discovered in \cite{TropGrassmannianScattering},
	$$\eta_{246} =s_{156}+s_{256}+s_{345}+s_{346}+s_{356}+s_{456}.$$
	Here the $s_{ijk}$'s are parameters, called Mandelstam invariants, as $\{i,j,k\}$ runs over all 3-element subsets of $\{1,\ldots, n\}$.
	
	This work had an empirical origin, going back to numerical computations in \cite{TropGrassmannianScattering}, for $m^{(3)}_6$.  First, when the kinematic data $(s)$ approaches the hyperplanes $\eta_{ijk}(s) = 0$ in the kinematic space, then invoking the CEGM formula one finds that the generalized amplitude $m^{(3)}_6$ develops a singularity.  Among the first examples that one encounters is the residue
	\begin{eqnarray}\label{eq:3n3split}
	\text{Res}\lbrack (m^{(3)}_6(\mathbb{I},\mathbb{I}))\rbrack_{\eta_{246} =0} & = &  \left(\frac{1}{\eta _{236}}+\frac{1}{\eta _{124}}\right) \left(\frac{1}{\eta _{256}}+\frac{1}{\eta _{146}}\right) \left(\frac{1}{\eta _{346}}+\frac{1}{\eta _{245}}\right).
	\end{eqnarray}
	That the residue is a product of not two, but \textit{three} terms, is a novel and intriguing behavior from a physical point of view; moreover, a close examination of the eight terms in the expansion of the product in Equation \eqref{eq:3n3split}, leads to identification of eight weakly separated collections, each of size three, all of which form weakly separated collections with $\{2,4,6\}$.  In particular, one derives that if $\{2,4,6\}$ and some triple $\{a,b,c\}$ are not weakly separated, that the residue $\text{Res}\lbrack (m^{(3)}_6(\mathbb{I},\mathbb{I}))\rbrack_{(\eta_{246},\eta_{abc}) = (0,0)}$ is identically zero.  Could this be explained?  Is there a more general statement?  Such questions were the starting point for this work. 
	
	From a physical point of view, the motivation for the present work can be summarized in the following prediction: if CEGM amplitudes $m^{(k)}_n$ have a physical meaning then weak separation becomes the correct generalization of the constraints that unitarity and locality impose on standard QFT regarding compatible poles.

	\subsection{Summary of Results}

	This paper has two main results.  In Section \ref{sec:Matroid subdivisions and matroidal blade arrangements}, Theorem \ref{thm: multisplit blade} describes the matroid subdivision on $\Delta_{k,n}$ which is induced when a single blade is placed on a vertex $e_J$ of $\Delta_{k,n}$.  Here it is shown that the subdivision is trivial precisely for vertices which correspond to frozen subsets, those which are cyclic intervals $J=\{i,i+1,\ldots, i+(k-1)\}$.  Our second main result is in Section 4, where Theorem \ref{thm: weakly separated matroid subdivision blade} shows that an arrangement of blades on the vertices of $\Delta_{k,n}$ induces a matroid subdivision if and only if the vertices define a weakly separated collection.  In the Appendix we list the numbers of maximal weakly separated collections through n=12, though some of the middle values of $k$ were impossible to achieve using our methods.

	\section{Blades and alcoved polytopes}\label{sec:Permutohedral blades and alcoved polytopes}
	Let $\mathcal{H}_{k,n}$ be the affine hyperplane in $\mathbb{R}^n$ where $\sum_{i=1}^n x_i=k$.  For integers $1\le k\le n-1$, denote by $\Delta_{k,n} = \left\{x\in \lbrack 0,1\rbrack^n:\sum_{j=1}^n x_j=k \right\}$ the $k^\text{th}$ hypersimplex of dimension $n-1$.  For a subset $J\subseteq \{1,\ldots, n\}$, denote $x_J = \sum_{j\in J} x_j$, and similarly for basis vectors, $e_J = \sum_{j\in J} e_j$.

	In \cite{OcneanuVideo}, A. Ocneanu introduced plates and blades, as follows.
	\begin{defn}[\cite{OcneanuVideo}]\label{defn:blade}
		A decorated ordered set partition $((S_1)_{s_1},\ldots, (S_\ell)_{s_\ell})$ of $(\{1,\ldots, n\},k)$ is an ordered set partition $(S_1,\ldots, S_\ell)$ of $\{1,\ldots, n\}$ together with an ordered list of integers $(s_1,\ldots, s_\ell)$ with $\sum_{j=1}^\ell s_j=k$.  It is said to be of type $\Delta_{k,n}$ if we have additionally $1\le s_j\le\vert S_j\vert-1 $, for each $j=1,\ldots, \ell$.  In this case we write $((S_1)_{s_1},\ldots, (S_\ell)_{s_\ell}) \in \text{OSP}(\Delta_{k,n})$, and we denote by $\lbrack (S_1)_{s_1},\ldots, (S_\ell)_{s_\ell}\rbrack$ the convex polyhedral cone in $\mathcal{H}_{k,n}$, that is cut out by the facet inequalities
		\begin{eqnarray}\label{eq:hypersimplexPlate}
		x_{S_1} & \ge & s_1 \nonumber\\
		x_{S_1\cup S_2} & \ge & s_1+s_2\nonumber\\
		& \vdots & \\
		x_{S_1\cup\cdots \cup S_{\ell-1}} & \ge & s_1+\cdots +s_{\ell-1}.\nonumber
		\end{eqnarray}
		These cones were studied as \textit{plates} by Ocneanu.
		Finally, the \textit{blade} $(((S_1)_{s_1},\ldots, (S_\ell)_{s_\ell}))$ is the union of the codimension one faces of the complete simplicial fan formed by the $\ell$ cyclic block rotations of $\lbrack (S_1)_{s_1},\ldots(S_\ell)_{s_\ell},\rbrack$, that is
		\begin{eqnarray}\label{eq: defn blade}
		(((S_1)_{s_1},\ldots, (S_\ell)_{s_\ell})) = \bigcup_{j=1}^\ell \partial\left(\lbrack (S_j)_{s_j},(S_{j+1})_{s_{j+1}},\ldots, (S_{j-1})_{s_{j-1}}\rbrack\right).
		\end{eqnarray}
	\end{defn}
\begin{rem}
	When there is no risk of confusion, depending on the context we shall use the notation $\lbrack (S_1)_{s_1},\ldots, (S_\ell)_{s_\ell}\rbrack$ for the cone in $\mathcal{H}_{k,n}$ or for the matroid polytope obtained by intersecting it with the hypersimplex $\Delta_{k,n}$.
\end{rem}

For an interesting use of decorated ordered set partitions that links combinatorics and convex geometry, see \cite{DonghyunKim}, where a proof was given for a conjecture of the author \cite{EarlyHStar} for a combinatorial interpretation of the $h^\star$-vector of generalized hypersimplices, in terms of decorated ordered set partitions.

\begin{figure}[h!]
	\centering
	\includegraphics[width=0.89\linewidth]{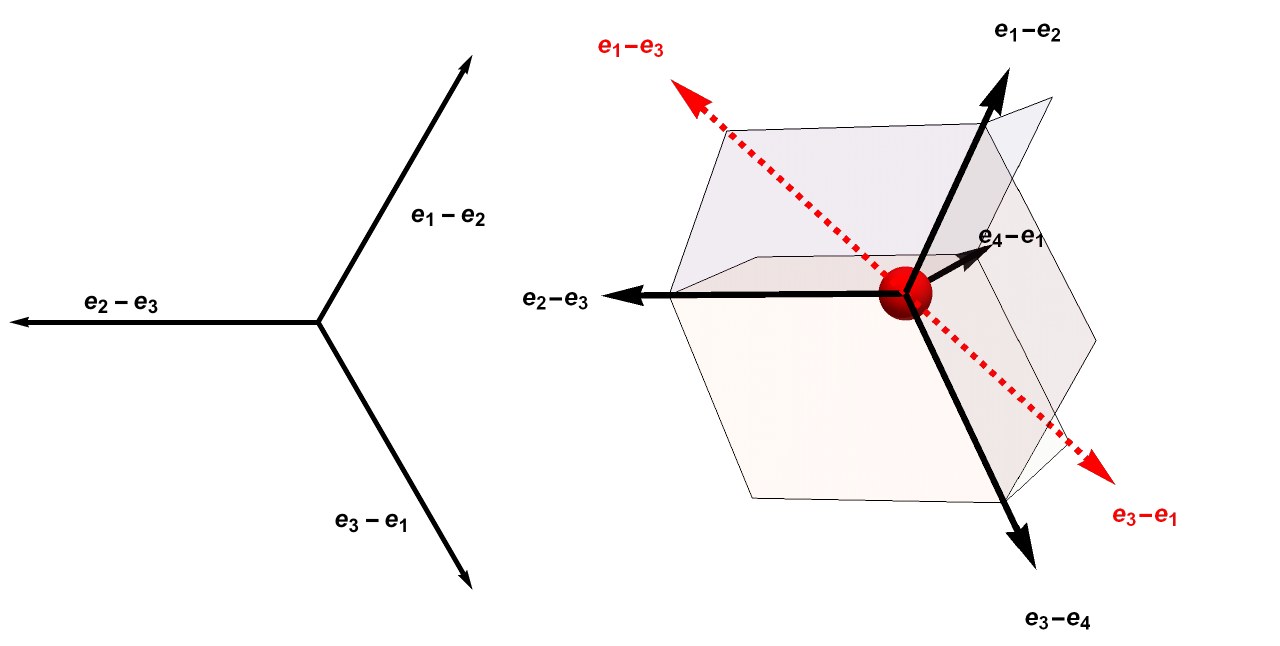}
	\caption{Left: the tripod $((1,2,3))$.  Right: the blade $((1,2,3,4))$ as a Minkowski sum of the two tripods $((1,2,3))$ and $((1,3,4))$.  See \cite{EarlyBlades} for the proof that such a decomposition exists in general.}
	\label{fig:blade3coordinates000}
\end{figure}

	\begin{prop}[\cite{EarlyBlades}]\label{prop: blades Minkowski sum}
		Removing one Minkowski summand at a time from each of the $\ell$ cyclic block rotations of $\lbrack S_1,\ldots, S_\ell\rbrack$ leads to a rewriting of Equation \eqref{eq: defn blade} using Minkowski sums, as
		\begin{eqnarray*}
			((S_1,\ldots, S_\ell)) & = & \bigcup_{1\le i<j\le \ell} \lbrack S_1,S_{2}\rbrack \boxplus\cdots  \lbrack \widehat{S_{i},S_{i+1}} \rbrack \boxplus\cdots \boxplus\widehat{\lbrack S_{j},S_{j+1\rbrack}}\boxplus\cdots \boxplus\lbrack S_\ell,S_1\rbrack,
		\end{eqnarray*}
		where we use the notation
		$$\lbrack S_i,S_j\rbrack = \left\{x\in \mathcal{H}_{0,n}:x_{S_i}\ge 0,\ x_{S_i\cup S_j}=0,\ x_\ell=0 \text{ whenever } \ell\not\in S_i\cup S_j \right\}$$
		for disjoint (nonempty) subsets $S_i,S_j$ of $\{1,\ldots, n\}$.
	\end{prop}
	\begin{rem}
	In Proposition \ref{prop: blades Minkowski sum} we saw that the blade $((S_1,\ldots, S_\ell))$, where by convention we put $s_1,\ldots, s_\ell=0$, can be conveniently expressed as a union of Minkowski sums.  Related expressions have appeared in the work \cite{DeltaAlgebra} of F. Cachazo, A. Guevara, S. Mizera and the author in the context of $k=2$ leading singularities, where each $\lbrack i,j\rbrack$ is replaced by a Grassmann-valued rational function
	$$u_{ij} = \frac{(\theta_i-\theta_j)(\theta_{i+n} - \theta_{j+n})}{x_i-x_j}.$$
	Here $\theta_1,\ldots, \theta_{2n}$ are anticommuting Grassmann variables.  Then the analog of the tripod $((a,b,c))$ is the sum 
	$$\Delta_{abc} = u_{ab} + u_{bc} + u_{ca}.$$
	The analog of a blade $((1,2,\ldots, n))$ is the product
	$$\Delta_{123}\Delta_{134}\cdots \Delta_{1,n-1,n} = \Delta_{T_1}\cdots \Delta_{T_{n-2}} = \frac{1}{(n-2)!}\left(u_{12}+u_{23} + \cdots + u_{n1}\right)^{n-2}$$
	$$= \sum_{1\le i<j\le n} u_{12}\cdots \widehat{u_{i,i+1}}\cdots \widehat{u_{j,j+1}}\cdots u_{n1},$$	
	where $\{T_1,\ldots, T_{n-2}\}$ is any (oriented) triangulation of a polygon with vertices cyclically labeled $\{1,\ldots, n\}$.  These identities are easy to check using basic properties of the $\Delta_{abc}$.  See \cite{DeltaAlgebra} for details; see also \cite{EarlyReiner} for connections to representation theory and topology.
\end{rem}
\subsection{Blades are tropical hypersurfaces}\label{sec: bladesInTropicalGeometry}
Recall now that a \textit{tropical polynomial function} $p:\mathbb{R}^n\rightarrow\mathbb{R}$ is the minimum of a finite set of linear functions,
$$p(x) = \text{min} \left\{f_1(x),f_2(x),\ldots, f_m(x)\right\},$$
whose graph $Z(p)$ is a piecewise-linear hypersurface in $\mathbb{R}^{n+1}$.  As such, the normal vector to $Z(p)$ changes direction only across a set of codimension 1 ``cracks'' in $\mathbb{R}^n$.  The collection of such cracks is called a \textit{tropical hypersurface} $\mathcal{V}(p),$ consisting of points in $\mathbb{R}^n$ such that the minimum value is achieved by at least two of the linear forms $f_i$.  For details and references of this and related constructions in tropical geometry, see for instance \cite{Sturmfels Speyer Tropical intro}.

In what follows we prove that any given blade $((\mathbf{S})) = ((S_1,\ldots, S_\ell))\subset \mathcal{H}_{0,n}$, it is equal to the tropical hypersurface defined by the tropical polynomial 
$$\mathfrak{h}_{((\mathbf{S}))}(x):=\text{min} \left\{L_1(x),\ldots, L_\ell(x)\right\},$$
where $L_i(x) = x_{S_{i+1}} +2x_{S_{i+2}}+\cdots +(\ell-1)x_{S_{i-1}}$.

\begin{prop}
	The blade $((1,2,\ldots, n))$ is isomorphic to the tropical hyperplane defined by the bends of the function $p:\mathcal{H}_{0,n}\rightarrow \mathbb{R}$,
	$$p(y) = \min\{y_1,\ldots, y_n\}.$$
\end{prop}

\begin{proof}
	Applying the change of variable
	$$(x_1,\ldots, x_n) = (y_2-y_1,y_3-y_2,y_4-y_3,\ldots, y_1-y_n)/n$$
	and taking $y_1+\cdots +y_n=0$, we obtain telescopically
	$$L_i(y) = -(y_{i+1}+y_{i+2}+\cdots +y_{i-1}-(n-1)y_{i})/n = y_i.$$
	
\end{proof}

\begin{prop}\label{prop: blade trop variety}
	For any ordered set partition $\mathbf{S}=(S_1,\ldots, S_\ell)$, the blade $$((\mathbf{S})) = ((S_1,S_2,\ldots,S_\ell))$$ is a tropical hypersurface in $\mathcal{H}_{0,n}$, defined by the tropical polynomial $h_{((\mathbf{s}))}(x)$, that is we have
	$$((\mathbf{S})) = \mathcal{V}\left(h_{((\mathbf{s}))}\right).$$
\end{prop}

\begin{proof}
	Suppose we are given a point $x\in \mathcal{H}_{0,n}$ that is simultaneously a minimum of some $L_i$ and $L_j$, with value say $c_x$.  Taking without loss of generality $1\le i<j\le \ell$, then 
	$$c_x = x_{S_{i+1}} + 2x_{S_{i+2}} + \cdots + (\ell-1)x_{S_{i-1}} = x_{S_{j+1}} + 2x_{S_{j+2}} + \cdots + (\ell-1)x_{S_{j-1}},$$
	or equivalently
	$$(\ell-(j-i))x_{S_i\cup\cdots S_{j-1}} = (j-i)x_{S_j\cup\cdots S_{i-1}}.$$
	But since in $\mathcal{H}_{0,n}$ we have $x_{S_i\cup\cdots \cup S_{j-1}} + x_{S_{j}\cup\cdots \cup S_{i-1}} = \sum_{i=1}^n x_i=0$, it follows that
	$$\ell x_{S_i\cup\cdots \cup  S_{j-1}} = (j-i)x_{S_i\cup S_{i+1}\cdots \cup S_{i-1}} = 0.$$
	Now, supposing that $c_x = L_i(x) = L_j(x)$ is a minimum value, then for all $p \in\{1,\ldots, n\}\setminus\{i,j\}$ we have
\begin{eqnarray*}
	x_{S_{p+1}} +2x_{ S_{p+2}}+\cdots + (\ell-1)x_{S_{p-1}}& \ge & c_x.
\end{eqnarray*}	
	Consequently, taking any $p\in\{i,\ldots, j-1\}$, then
	\begin{eqnarray*}
		(x_{S_{i+1}} +2x_{S_{i+2}}+\cdots +(\ell-1)x_{S_{i-1}}) - (x_{S_p} +x_{S_p \cup S_{p+1}}\cdots + x_{S_p \cup S_{p+1}\cup\cdots \cup S_{p -2}}) & \ge & 0\\
		(\ell-(p-i))x_{S_1\cup \cdots \cup S_{p-1}} - (p-i)x_{S_{p}\cup\cdots \cup S_{\ell}}& \ge  & 0,
	\end{eqnarray*}
	hence
	$$\ell x_{S_1\cup \cdots \cup S_{p-1}} \ge (p-i)x_{S_1\cup\cdots\cup S_\ell}=0,$$
	and so
	$$x_{S_i\cup\cdots \cup S_{p-1}}\ge 0.$$
	Similarly we find $x_{S_j\cup\cdots \cup S_{q-1}}\ge 0$ for each $q=j,j+1,\ldots, i-1$.  Repeating for each $1\le i<j\le n$ and putting everything together, we recover the blade $((S_1,\ldots, S_k))$, as 
	\begin{eqnarray*}
		&&\bigcup_{1\le i<j\le k} \left\{x: x_{S_i}\ge 0,x_{S_i\cup S_{i+1}}\ge 0,\ldots, \ x_{S_j}\ge 0,x_{S_j\cup S_{j+1}}\ge 0,\ldots; \ x_{S_i\cup\cdots \cup S_{j-1}} = 0= x_{S_j\cup\cdots \cup S_{i-1}} \right\}\\
		& = & \bigcup_{1\le i<j\le k} \lbrack S_i,S_{i+1},\ldots, S_{j-1}\rbrack \boxplus\lbrack S_{j},S_{j+1},\ldots, S_{i-1}\rbrack\\
		& = & \bigcup_{i=1}^k \partial\left(\lbrack S_i,S_{i+1}\ldots, S_{i-1}\rbrack\right).
	\end{eqnarray*}
	Here the last line is in agreement with Definition \ref{defn:blade}.
\end{proof}

	In what follows, we shall use the notation $((1,2,\ldots, n))_p$ for the translation of the blade $((1,2,\ldots, n))$ by $p\in\mathbb{R}^n$.
	
	The matroids which encode the vertices of the matroid polytopes $\lbrack (S_1)_{s_1},\ldots, (S_\ell)_{s_\ell}\rbrack \cap \Delta_{k,n}$ have appeared under various names, such as for instance \textit{nested} matroids and \textit{Schubert} matroids.  For a comprehensive discussion see for instance the recent work \cite{Schroter} and references therein.
	
	It is not hard to check (or see \cite{EarlyBlades}) that the $\ell$ cyclic rotations
	$$\lbrack (S_1)_{s_1},(S_2)_{s_2},\ldots, (S_\ell)_{s_\ell}\rbrack,\ \lbrack (S_2)_{s_2},(S_3)_{s_3}\ldots, (S_1)_{s_1}\rbrack,\ldots,\lbrack (S_{\ell})_{s_{\ell}},(S_{\ell+1})_{s_{\ell+1}}\ldots, (S_{\ell-1})_{s_{\ell-1}}\rbrack,$$
	where $s_1+\cdots + s_\ell=k$, form a complete (simplicial) fan in $\mathcal{H}_{k,n}$.  Moreover, in the case that the $S_i$ are all singletons, then $\ell=n$ and $(((S_1)_{s_1},\ldots, (S_n)_{s_n}))$ is the set of facets of (a translation of) the normal fan to a Weyl alcove, namely the simplex in $\mathcal{H}_{0,n}$ where
	$$y_{j_1}\le y_{j_2}\le\cdots \le y_{j_n} \le y_{j_1}+1,$$
	for some permutation $(j_1,\ldots, j_n)$ of $(1,\ldots, n)$.  Consequently, we shall say in particular that the blade $((1,2,\ldots, n))$ \textit{induces} the normal fan to the fundamental Weyl alcove which is characterized by 
	$$y_{1}\le y_2\le\cdots \le x_{n}\le y_1+1.$$

\subsection{Relations to alcoved polytopes}
	Recall that we have fixed the cyclic order $(1,2,\ldots, n)$.
	\begin{defn}[\cite{AlcovedPolytopes}]
		A polytope in $\mathbb{R}^{n-1}$ is said to be \textit{alcoved} if its facet inequalities are of the form $b_{ij}\le y_i-y_j \le  c_{ij}$ for some collection of integer parameters $b_{ij}$ and $c_{ij}$.
	\end{defn}
	As noted in \cite{AlcovedPolytopes}, any alcoved polytope comes with a natural triangulation into Weyl alcoves.

	\begin{defn}
	A polytopal subdivision $\Pi_1\cup\cdots \cup \Pi_t = \Delta_{k,n}$ of a hypersimplex $\Delta_{k,n}$ is said to be a matroid polytope if every maximal cell is a matroid polytope; it is moreover \textit{planar} if every maximal cell $\Pi_i$ is a positroid polytope, that is, its facets are given by equations $x_{i}+x_{i+1}+\cdots+x_{i+m}=r_{i,i+m}$ for some integers $r_{i,i+m}$, where $i\in\{1,\ldots, n\}$ and $1\le m \le n-2$, where the indices are assumed to be cyclic. 
\end{defn}
In light of the change of variables $y_1 = x_1,y_2 = x_{1}+x_2,\ldots, y_{n-1} = x_1+\cdots +x_{n-1}$ on $\mathbb{R}^{n-1}$, we shall abuse terminology and call the maximal cells of a planar polytopal subdivision alcoved polytopes.

	In this paper, unless otherwise stated, we consider exclusively arrangements of the blade defined by a single cyclic order, $((1,2,\ldots, n))$, not its reflection $((1,n,n-1,\ldots, 2))$.  Nonetheless we make one basic observation.
	
	\begin{prop}\label{prop: subdivisionAlcoved polytopes}
		Any arrangement of the blades $((1,2,\ldots, n))$ and $((1,n,n-1,\ldots, 2))$ on the vertices of $\Delta_{k,n}$ induces a polytopal subdivision where the maximal cells are alcoved polytopes.  The subdivision is refined by the alcove triangulation.
	\end{prop}
		\begin{figure}[h!]
	\centering
	\includegraphics[width=0.7\linewidth]{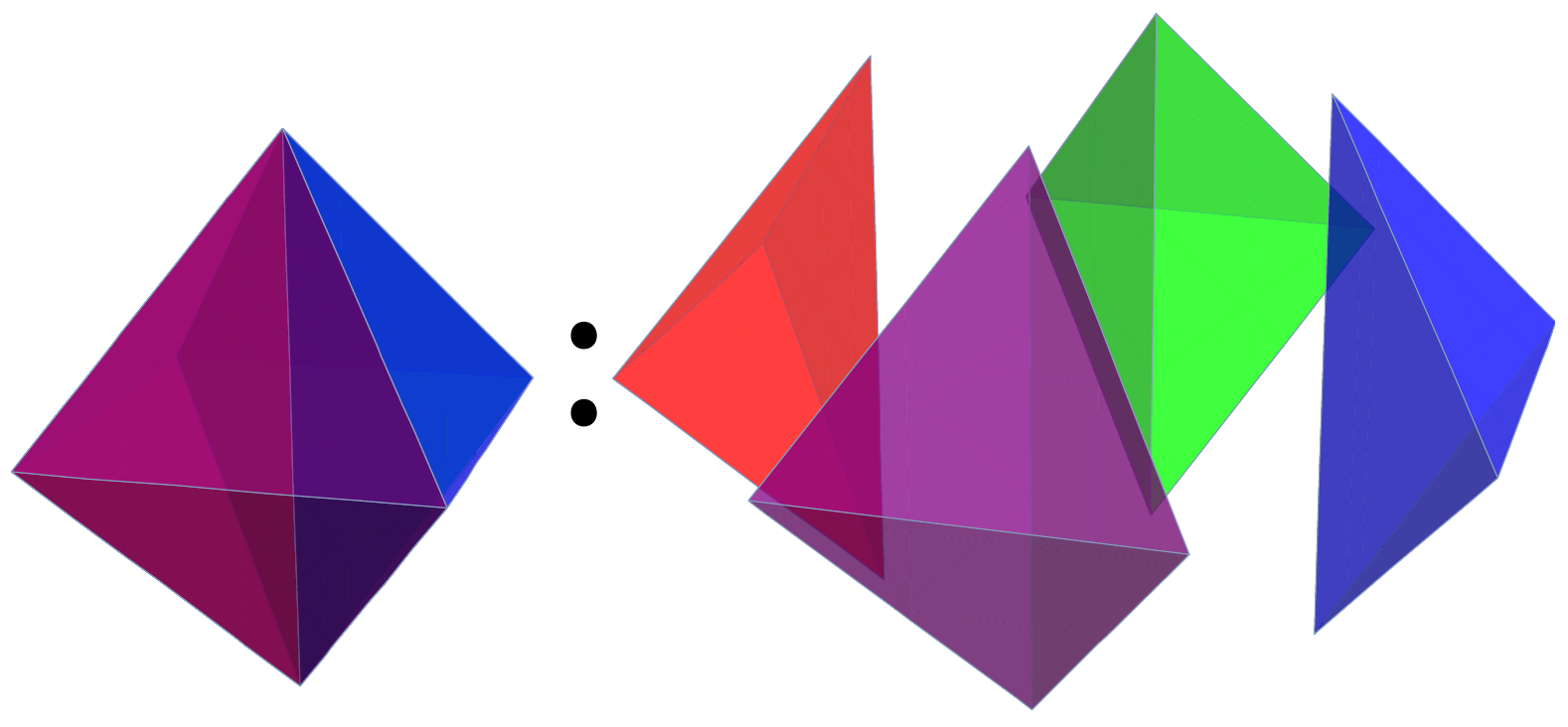}
	\caption{Alcove triangulation of the octahedron $\Delta_{2,4}$}
	\label{fig:lemonslices}
\end{figure}
	\begin{example}\label{example: octahedral nonmatroidal blade arrangement}
		The basic example of a non-matroidal blade arrangement is
		$$\{((1,2,3,4))_{e_{13}}, ((1,2,3,4))_{e_{24}}\}.$$
		Looking forward to Definition \ref{defn:weakly separated}, note that the sets $\{13,24\}$ do not form a weakly separated collection.  See our second main result, Theorem \ref{thm: weakly separated matroid subdivision blade}, for the general criterion for matroidal blade arrangements in terms of weakly separated collections.

		Upon restriction to the second hypersimplex $\Delta_{2,4}$ we obtain the intersections
		$$((1,2,3,4))_{e_{13}}\cap \Delta_{2,4} = ((14_1 23_1))\cap \Delta_{2,4}$$
		$$((1,2,3,4))_{e_{24}}\cap \Delta_{2,4} = ((12_1 34_1))\cap \Delta_{2,4},$$
		which together induce the alcove triangulation of the octahedron into $A(1,3)=4$ simplices, where the $A(k-1,n-1)$ is the Eulerian number which counts the number of permutations of $\left\{1,\ldots, n-1\right\}$ with $k-1$ descents.  As these each tetrahedron has an edge which is not a root, the subdivision is not a matroid subdivision.  See Figure \ref{fig:lemonslices}.

	\end{example}
	From Proposition \ref{prop: cyclically nonplanar finest refinement}, it will follow that the maximal cells found in an arbitrary blade arrangement are in general not matroid polytopes as they may have edges which are not roots $e_i-e_j$, but rather contain sums of orthogonal roots, for instance $e_i-e_j+e_k-e_\ell$.  In the ``worst case scenario'' the maximal cells may be $(n-1)$-dimensional simplices, the Weyl alcoves, whose vertices are among the vertices of the hypersimplex.  In particular, in the case when copies of $((1,2,\ldots, n))$ are placed on enough vertices of $\Delta_{k,n}$, then the induced subdivision of $\Delta_{k,n}$ is the alcove triangulation \cite{AlcovedPolytopes}.

	We shall need the map on affine hyperplanes $\mathcal{H}_{r,n}\subset \mathbb{R}^n$ induced by 
	$$(x_1,\ldots, x_n) \mapsto (y_1,y_2,\ldots,y_n) = (x_1,x_{12},\ldots, x_{12\cdots n}).$$
	\begin{prop}\label{prop: cyclically nonplanar finest refinement}
		The common refinement of all planar matroid subdivisions coincides with the alcove triangulation of the hypersimplex, see \cite{AlcovedPolytopes}.
	\end{prop}

	\begin{proof}
		We first check that the alcove triangulation is a common refinement of the planar matroid subdivisions, and then we show that it is the unique common refinement.
		
		For a fixed cyclic order, all maximal cells can be pulled back to the tessellation of $\mathbb{R}^n\slash (1,1,\ldots, 1)\mathbb{R}$ with Weyl alcoves, since any planar affine hyperplane cutting through the hypersimplex $\Delta_{k,n}$ has the form 
		$$y_j-y_i = x_{i+1}+x_{i+1}+\cdots +x_j = r_{ij}$$
		for some $i\not=j$, and for some integer $1\le r_{ij}\le (j-i)-1$.  

		To conclude the proof, simply observe that the set of two-block planar matroid polytopes, that is the planar matroid 2-splits, already have common refinement the alcove triangulation.
	\end{proof}

	\begin{example}We compute the facet inequalities which define the 11 simplices in the alcove triangulation of $\Delta_{2,5}$, and recover the adjacency graph for the alcove triangulation in Figure \ref{fig:adjacency-graph-d25} from \cite{AlcovedPolytopes}.
		\begin{enumerate}
			\item Five simplices having two facets in the interior of $\Delta_{2,5}$:
			$$\lbrack 12_1 345_1\rbrack \cap \lbrack 23_1 451_1\rbrack,\lbrack 23_1 451_1\rbrack \cap \lbrack 34_1 512_1\rbrack,\ldots, \lbrack 51_1 234_1\rbrack \cap \lbrack 12_1 345_1\rbrack,$$
			that is, respectively,
			$$\left\{x\in\Delta_{2,5}: x_{12}\ge 1, \ x_{23}\ge 1 \right\}, \left\{x\in\Delta_{2,5}: x_{23}\ge 1, \ x_{34}\ge 1 \right\},\ldots, \left\{x\in\Delta_{2,5}: x_{51}\ge 1, \ x_{12}\ge 1 \right\}.$$
			\item Five simplices with three facets in the interior of $\Delta_{2,5}$:
			$$\lbrack 451_1 23_1\rbrack \cap \lbrack 12_1 345_1\rbrack\cap \lbrack 34_1 512_1\rbrack,\lbrack 512_1 34_1\rbrack \cap \lbrack 23_1 451_1\rbrack\cap \lbrack 45_1 123_1\rbrack,\ldots, \lbrack 345_1 12_1\rbrack \cap \lbrack 51_1 234_1\rbrack\cap \lbrack 23_1 451_1\rbrack,$$
			that is, respectively,
			$$\left\{x\in\Delta_{2,5}:x_{451}, x_{12}, x_{234}\ge 1\right\}, \left\{x\in\Delta_{2,5}:  x_{512}, x_{23}, x_{345}\ge 1\right\},\ldots, \left\{x\in\Delta_{2,5}: x_{345}, x_{51}, x_{123}\ge 1\right\}$$
			\item There is one simplex in $\Delta_{2,5}$ with all five facets in the interior of $\Delta_{2,5}$:
			$$\left\{x\in\Delta_{2,5}: x_{123}\ge 1,\ x_{234}\ge 1,\ x_{345}\ge 1,\ x_{451}\ge 1,\ x_{512}\ge 1\right\}.$$
		\end{enumerate}
		Then, translating the blade $((1,2,3,4,5))$ to the five vertices
		$$e_{35},e_{41},e_{52},e_{13},e_{24}\in\Delta_{2,5}$$
		one induces the full alcove triangulation of $\Delta_{2,5}$.
		
		\begin{figure}[h!]
			\centering
			\includegraphics[width=0.8\linewidth]{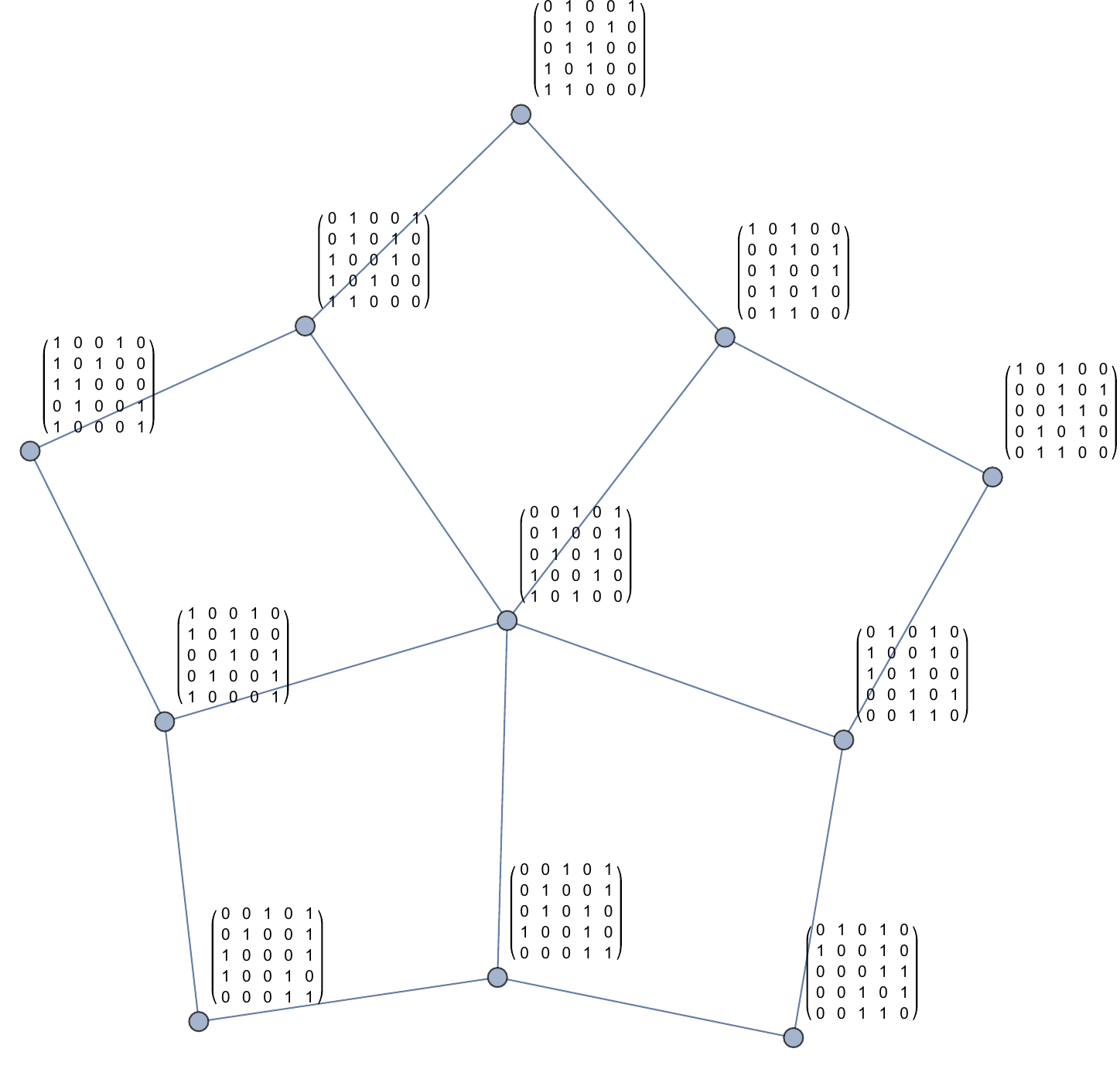}
			\caption{Dual graph to the alcove triangulation of $\Delta_{2,5}$.  Each node represents a simplex in the alcove triangulation of $\Delta_{2,5}$ and is labeled by its vertices in $\Delta_{2,5}$, listed as rows.  Two simplices are joined by an edge when they share a facet.  Thus we have recovered the construction using circuits in \cite{AlcovedPolytopes}, Figure 2.}
			\label{fig:adjacency-graph-d25}
		\end{figure}
		For instance, looking ahead towards Theorem \ref{thm: multisplit blade},
		$$((1,2,3,4,5))_{e_{35}}\cap \Delta_{2,5} = ((1_0 2_0 3_1 4_0 5_1))\cap \Delta_{2,5}=((123_1 45_1))\cap \Delta_{2,5},$$
		where $((123_1 45_1))$ is the affine plane where $x_{123}=1$, which cuts through the interior of the hypersimplex $\Delta_{2,5}$.
	\end{example}

	\section{Matroid subdivisions and matroidal blade arrangements}\label{sec:Matroid subdivisions and matroidal blade arrangements}
	
	\begin{defn}
		Let $d\ge 1$.  A $d$-split of an $m$-dimensional polytope $P$ is a coarsest subdivision $P = P_1\cup\cdots \cup P_d$ into $m$-dimensional polytopes $P_i$, such that the polytopes $P_i$ intersect only on their common facets, and such that 
		$$\text{codim}(P_1\cap\cdots\cap P_d) = d-1.$$  If $d$ is not specified, then we shall use the term multi-split.  
	\end{defn}
	Note that for a subset $I\subset\{1,\ldots, n\}$ with $\vert I\vert=k$, the element $e_I := \sum_{i\in I} e_i$ is a vertex of $\Delta_{k,n}$.  
	
	For any given cyclic order $\sigma=(\sigma_1, \sigma_2, \cdots, \sigma_n)$, each $k$-element subset $I= \{i_1,\ldots, i_k\}\subset \{1,\ldots, n\}$ gives rise to a translation 
	$$((\sigma_1,\sigma_2,\ldots, \sigma_n))_{e_I}.$$
	of the blade $((\sigma_1,\sigma_2,\ldots, \sigma_n))$ by the vector $e_I$.  When no cyclic order or translation is specified we put simply $((1,2,\ldots, n))$.
	
	Evidently a polyhedral subdivision $P=P_1\cup\cdots \cup P_m$ into convex polyhedra $P_i$ is uniquely determined by the union of the convex spans of its internal facets; indeed, in the case when $P=\Delta_{k,n}$, then $P_1,\ldots, P_m$ are the $m$ cyclic block rotations of the (nested) matroid polytope $\lbrack (S_1)_{s_1},\ldots, (S_m)_{s_m}\rbrack$ if and only if $P$ is induced by the blade $(((S_1)_{s_1},\ldots, (S_m)_{s_m}))$.
	
	\begin{prop}\label{prop: nestedmatroids from blade}
		The maximal cells in the matroid subdivision induced by a blade $$(((S_1)_{s_1},\ldots, (S_\ell)_{s_\ell}))$$ are the $\ell$ matroid polytopes
		$$\lbrack (S_1)_{s_1},(S_2)_{s_2},\ldots, (S_\ell)_{s_\ell}\rbrack,\lbrack (S_2)_{s_2},(S_3)_{s_3},\ldots, (S_1)_{s_1}\rbrack ,\ldots, \lbrack (S_\ell)_{s_\ell},(S_1)_{s_1},\ldots, (S_{\ell-1})_{s_{\ell-1}}\rbrack.$$
		Moreover, a maximal cell $\lbrack (S_i)_{s_i},\ldots, (S_{i-1})_{s_{i-1}}\rbrack$ intersects the facet $x_j=1$ of $\Delta_{k,n}$ if and only if $j$ is not in the last block $S_{i-1}$.
	\end{prop}
	
	It is well known \cite{Hiroshi} see also \cite{HermannJoswig} that the splits of the second hypersimplex $\Delta_{2,n}$ are in bijection with 2-block set partitions $\{S,S^c\}$ of $\{1,\ldots, n\}$ such that $1<\vert S\vert < n-1$; a direct translation to our terminology and notation, splits of $\Delta_{2,n}$ are \textit{induced} by blades $((S,T))$, where $\{S,T\}$ is a set partition of $\{1,\ldots, n\}$ with $\vert S \vert,\vert T\vert \ge 2$.

	For the second hypersimplices $\Delta_{2,n}$, there is a compatibility condition for splits to induce matroid subdivisions, see \cite{Hiroshi}, and \cite[Proposition 5.4]{HermannJoswig}, by specializing to $k=2$.
Via the correspondence between the split $(S,S^c)$ and the blade $((S,S^c))$ which induces it we have the following characterization of matroid subdivisions of $\Delta_{k,n}$ which arise by refining collections of compatible 2-splits.
	\begin{cor}\label{cor:blade splits}
		A pair of blades $(((S)_1,(S^c)_1))$ and $(((T)_1,(T^c)_1))$ induces a matroid subdivision of $\Delta_{2,n}$ if and only if at least one of the following four intersections is empty:
		$$S\cap T,\ S\cap T^c,\ S^c\cap T,\ S^c\cap T^c.$$
	\end{cor}
	In the case when both blades are planar (i.e. $S$ and $T$ are both intervals on the circle), then split compatibility of a collection of subsets is equivalent to asking that the corresponding collection of 2-element subsets be weakly separated.  See Theorem \ref{thm: weakly separated matroid subdivision blade}.

\begin{defn}
	An arrangement of blades $\{((1,2,\ldots, n))_{e_{I_1}},\ldots, \{((1,2,\ldots, n))_{e_{I_m}}\}$ on the vertices $e_{I_1},\ldots, e_{I_m}$ of the hypersimplex $\Delta_{k,n}$ is \textit{matroidal} if every chamber in the induced subdivision is a matroid polytope.
\end{defn}
In Theorem \ref{thm: multisplit blade}, our first main result, with respect to the natural order $((1,2,\ldots, n))$, given a vertex $e_J$ we give the explicit formula for the decorated ordered set partition $((S_1)_{s_1},\ldots, (S_\ell)_{s_\ell})\in\text{OSP}(\Delta_{k,n})$ such that on the hypersimplex the translated blade and the hypersimplicial blade $(((S_1)_{s_1},\ldots, (S_\ell)_{s_\ell}))$ coincide, that is
$$((1,2,\ldots, n))_{e_J} \cap \Delta_{k,n} = (((S_1)_{s_1},\ldots, (S_\ell)_{s_\ell}))\cap \Delta_{k,n}.$$
This phenomenon is illustrated in Figure \ref{fig:hexagon-blade-arrangement}.

	\begin{thm}\label{thm: multisplit blade}
		Let $e_I$ be a vertex of $\Delta_{k,n}$ and fix an $n$-cycle, say without loss of generality $\sigma = (1,2,\ldots, n)$.  Then, the translated blade $((1,2,\ldots, n))_{e_{I}}$ induces a multi-split matroid subdivision of $\Delta_{k,n}$, with $\ell$ maximal cells, that is trivial precisely when the subset $I$ coincides with a cyclic interval in $\{1,\ldots, n\}$: 
		$$\left(((1,2,\ldots, n))_{e_{I}}\right)\cap \Delta_{k,n} = \left(((S_1)_{s_1},(S_2)_{s_2},\ldots, (S_\ell)_{s_\ell})\right)\cap \Delta_{k,n},$$
		where $((S_1)_{s_1},\ldots, (S_\ell)_{s_\ell})\in \text{OSP}(\Delta_{k,n})$, constructed explicitly in the proof, satisfies the property that $\ell$ equals the number of cyclic intervals in $I$.  In particular, the blade induces the trivial matroid subdivision, if and only if $I$ is a cyclic interval.
	\end{thm}
	
	\begin{proof}
		Suppose first that we are given one of the $n$ cyclically contiguous subsets $I$, which, without loss of generality, we take to be $I=\{1,2,\ldots, k\}$.  We need to determine the intersection of 
		$$((1,2,\ldots, n))_{e_{12\cdots k}} =((1_1,2_1,\ldots, k_1,(k+1)_0,\ldots, n_0))$$
		with the hypersimplex $\Delta_{k,n}$.

		Note that the $\binom{n}{2}$ cones of $((1,2,\ldots, n))$ are Cartesian products of pointed cones which are in hyperplanes in bijection with partitions $J\cup J^c = \{1,\ldots, n\}$, where $J\subsetneq \{1,\ldots, n\}$ is a cyclically contiguous (nonempty) subset.  We have therefore two cases for each decomposition $J \cup  J^c = \{1,2,\ldots, n\}$.  Without loss of generality, let us assume that $J\cap \{1,\ldots, k\}\not=\emptyset $.  The cases are as follows.
		\begin{enumerate}
			\item $J$ is cyclically contiguous with $J\subseteq \{1,\ldots, k\}$.  Then, from the definition of $((1,2,\ldots, n))_{e_I}$, we have $x_J = \vert J\vert$, and since we are in a hypersimplex where $x_j\in\lbrack 0,1\rbrack$, the affine hyperplane 
			$$\mathcal{H}_{\vert J\vert,n}=\left\{x\in\mathbb{R}^n: x_J=\vert J\vert,\  x_{J^c} = k-\vert J\vert\right\}$$
			 does not cut through the interior of $\Delta_{k,n}$, hence neither does the sheet of the blade and so the induced subdivision is trivial.
			\item $J$ is cyclically contiguous with $J\cap \{1,\ldots, k\}$ and $J \cap \{1,\ldots, k\}^c$ both nonempty.  We have $x_{J\cap \{1,\ldots, k\}} \ge \vert J\cap \{1,\ldots, k\} \vert$, and again the cone does not cut through the interior of $\Delta_{k,n}$, and the induced subdivision is trivial.
		\end{enumerate}
		Assuming now that $I$ is not cyclically contiguous, then let
		$$I=\{i_1,i_1+1,\ldots, i_1+({\lambda_1}-1)\} \cup \{i_2,i_2+1,\ldots, i_2+({\lambda_2}-1)\}\cup \cdots\cup \{i_{\ell},i_{\ell}+1,\ldots, i_\ell + (\lambda_{\ell}-1)\},$$
		be its decomposition into cyclic intervals, so that we have $\lambda_1+\cdots+ \lambda_\ell = k$.   Denote $I_j = \{i_j,i_j+1,\ldots, i_j+({\lambda_j}-1)\}$ where without loss of generality we assume that $1\in I_1$.  Let $(C_1,\ldots, C_\ell)$ be the interlaced complement to the intervals $I$, so that we have the concatenation
		$$(1,2,\ldots, n) = (C_1,I_1,C_2,I_2,\ldots, C_\ell,I_\ell).$$
		In other words, the $I_j$ are the positions of consecutive one's, and the $C_j$ are the positions of consecutive zero's.

		We claim that 
		$$\left(((1,2,\ldots, n))_{e_{I}}\right)\cap \Delta_{k,n} = \left(((S_1)_{s_1},(S_2)_{s_2},\ldots, (S_\ell)_{s_\ell})\right)\cap \Delta_{k,n},$$
		where $(\mathbf{S}_{\mathbf{s}}) = ((S_1)_{s_1},\ldots, (S_\ell)_{s_\ell}) \in \text{OSP}(\Delta_{k,n})$ is the decorated ordered set partition defined by $(S_j,s_j) = (I_j\cup C_j , \vert I_j \vert )$. 
		
		To see this, again take a nontrivial set partition $\{J,J^c\}$ of $\{1,\ldots, n\}$ into cyclically consecutive sets $J$ and $J^c$, where $J = \{j_1,\ldots, j_t\}$, say, with $t\ge 1$.  We study the intersection of $((1,2,\ldots, n))_{e_I}$ with the hyperplane $U_I = \left\{x\in \mathcal{H}_{k,n}: x_J = \vert J\cap I\vert \right\}$.  Here we note that $U_I\cap \Delta_{k,n} \simeq \Delta_{\vert J\cap I\vert,\vert J\vert}\times \Delta_{k-\vert J\cap I\vert,n-\vert J\vert}$.  
		
		The intersection of the blade with $U_I$ is a (translated) Minkowski sum of (orthogonal) simplicial cones: 
		$$((1,2,\ldots, n))_{e_I}\cap U_I  = \lbrack j_1,\ldots, j_t\rbrack_{e_{J\cap I}} \boxplus \lbrack j_{t}+1,j_{t}+2\ldots, j_{1}-1\rbrack_{e_{J^c\cap I}}.$$
		The first factor $\lbrack j_1,\ldots, j_t\rbrack_{e_{J\cap I}}$ is characterized by the facet inequalities
		\begin{eqnarray*}
			x_{j_1} & \ge & r_{1}\\
			x_{j_1j_2} & \ge & r_{1}+r_{2}\\
			& \vdots &\\
			x_{j_1\cdots j_{t-1}} & \ge  & r_{1}+r_{2}+\cdots +r_{t-1}\\
			x_{j_1\cdots j_t} & = & r_{1}+r_{2}+\cdots + r_{t-1}+r_t\\
			& = & \vert J\cap I\vert,
		\end{eqnarray*}
		where $r_{\ell} = 1$ if $j_\ell \in J\cap I$, and otherwise if $j_\ell \in J\cap I^c = J\cap (C_1\cup\cdots\cup C_\ell)$ then $r_{\ell} = 0$.  Now whenever $r_\ell=0$, the corresponding inequality 
		$$x_{j_1}+\cdots +x_{j_\ell} \ge r_1+\cdots +r_\ell$$
		is implied by the line above it and is redundant; thus, in the decorated ordered set partition $j_{\ell}$ joins the (possibly singleton) block containing $j_{\ell-1}$.  Repeating this argument for all $j_\ell$ such that $r_\ell=0$ gives an ordered set partition $(S_1,\ldots, S_u)$ of $J$, where $S_j = I_j\cup C_j$ consists of a consecutive interval of labels $I_j$ corresponding to the positions of 1 in the vector $e_I$, and a consecutive interval $C_j$ consisting of the positions of $0$ in the vector $e_I$, and we have $1\le u\le \ell-1$.  Moreover, the block $S_j$ is accompanied by  $s_j=\vert S_j \cap I\vert  = \vert I_j\vert$.  Repeating the argument for $J^c$ and all other set partitions $\{J,J^c\}$ of $\{1,\ldots, n\}$ into cyclically consecutive intervals, then in the decorated ordered set partition notation we find the expression $((S_1)_{s_1},\ldots, (S)_{s_\ell})$.

		Consequently, restricting further to the hypersimplex $\Delta_{\vert J\cap I\vert,J}\subset U_I$ gives 
		$$\lbrack j_1,\ldots, j_t\rbrack_{e_{J\cap I}} \cap\Delta_{\vert J\cap I\vert,J} = \lbrack (S_1)_{s_1},\ldots, (S_u)_{s_u}\rbrack \cap\Delta_{\vert J\cap I\vert,J},$$
		where we denote by
		$$\Delta_{a,J} = \left\{\sum_{j\in J}t_j e_j\in\lbrack 0,1\rbrack^{n} :\sum_{j\in J}t_j = a\right\}$$
		the hypersimplex, isomorphic to $\Delta_{a,\vert J\vert}$, that is supported in the affine subspace indexed by $J$.

		Finally, translating from Theorem 3.14 of \cite{SchroterMultisplits}, any (nontrivial) multi-split of the hypersimplex $\Delta_{k,n}$ has for its maximal cells the $\ell\ge 2$ cyclic block rotations of some nested matroid $\lbrack (S_1)_{s_1},\ldots, (S_\ell)_{s_\ell}\rbrack,$ where $( (S_1)_{s_1},\ldots, (S_\ell)_{s_\ell})\in\text{OSP}(\Delta_{k,n})$.  This is precisely the condition that characterizes the ($\ell$-split) matroid subdivision induced by the blade 
		$$(((S_1)_{s_1},\ldots, (S_\ell)_{s_\ell})).$$
		Noting that $\ell$ is the number of cyclic intervals of 1's completes the proof.
	\end{proof}

\begin{example}
	For $(k,n) = (4,8)$, 
	$$((1,2,3,4,5,6,7,8))_{e_{1247}}\cap \Delta_{4,8} = ((1_1 2_1 3_0 4_1 5_0 6_0 7_1 8_0))\cap \Delta_{4,8} = ((128_2 34_1 567_1))\cap \Delta_{4,8}.$$
\end{example}
	\begin{example}\label{example: matroid subdivision 4 coords}
	The four chambers of the matroid subdivision induced in $\Delta_{4,8}$ by the blade $((12_1 34_1 56_1 78_1 ))$, where
	$$((1,2,\ldots, 8))_{e_{2468}}\cap \Delta_{4,8} = ((12_1 34_1 56_1 78_1))\cap \Delta_{4,8},$$
	are cut out by the facet inequalities respectively
	\begin{center}
		$$x_{12}\ge 1,\ x_{1234}\ge 1+1,\ x_{123456}\ge 1+1+1$$
		$$x_{34}\ge 1,\ x_{3456}\ge 1+1,\ x_{345678}\ge 1+1+1$$
		$$x_{56}\ge 1,\ x_{5678}\ge 1+1,\ x_{567812}\ge 1+1+1$$
		$$x_{78}\ge 1,\ x_{7812}\ge 1+1,\ x_{781234}\ge 1+1+1.$$
	\end{center}
	These chambers project under the map $(x_1,\ldots, x_8)\mapsto (x_{12},x_{34},x_{56},x_{78})$ onto the four chambers in the subdivision $((1_1 2_1 3_1 4_1))$ of the second dilation of the octahedron $\left\{y\in \lbrack 0,2\rbrack^4: \sum_{i=1}^4 y_i=4 \right\}$ in Figure \ref{fig:4-9-2019-octahedron-in-blade}, with facet inequalities respectively
	\begin{center}
		$$y_{1}\ge 1,\ y_{12}\ge 1+1,\ y_{123}\ge 1+1+1$$
		$$y_{2}\ge 1,\ y_{23}\ge 1+1,\ y_{234}\ge 1+1+1$$
		$$y_{3}\ge 1,\ y_{34}\ge 1+1,\ y_{341}\ge 1+1+1$$
		$$y_{4}\ge 1,\ y_{41}\ge 1+1,\ y_{412}\ge 1+1+1.$$
	\end{center}
	Here we have used different variables to emphasize that one set of polytopes is in $\Delta_{4,8}$, while the other set is in the second dilation of the octahedron $\Delta_{2,4}$.
\end{example}
\begin{figure}[h!]
	\centering
	\includegraphics[width=0.4\linewidth]{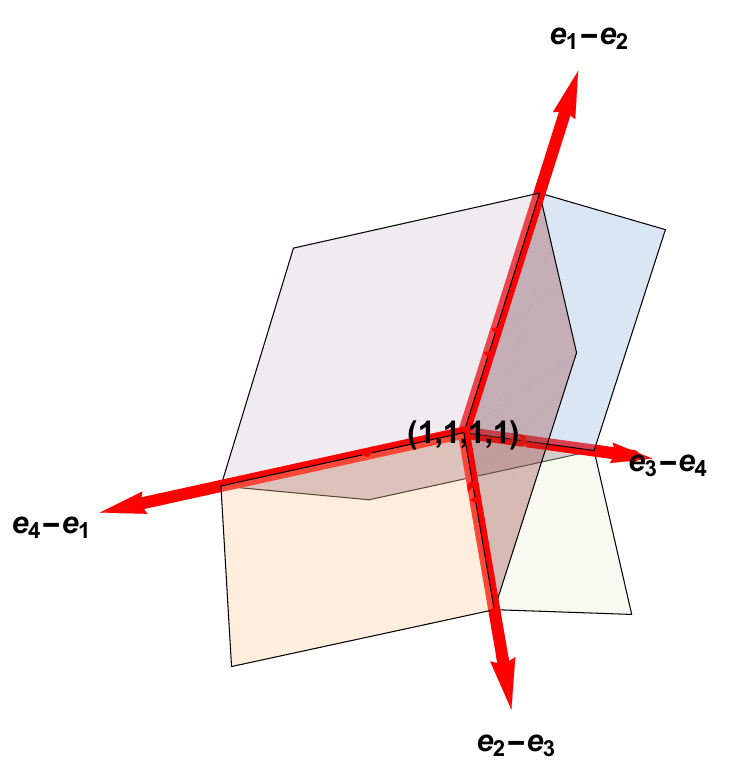}
	\includegraphics[width=0.4\linewidth]{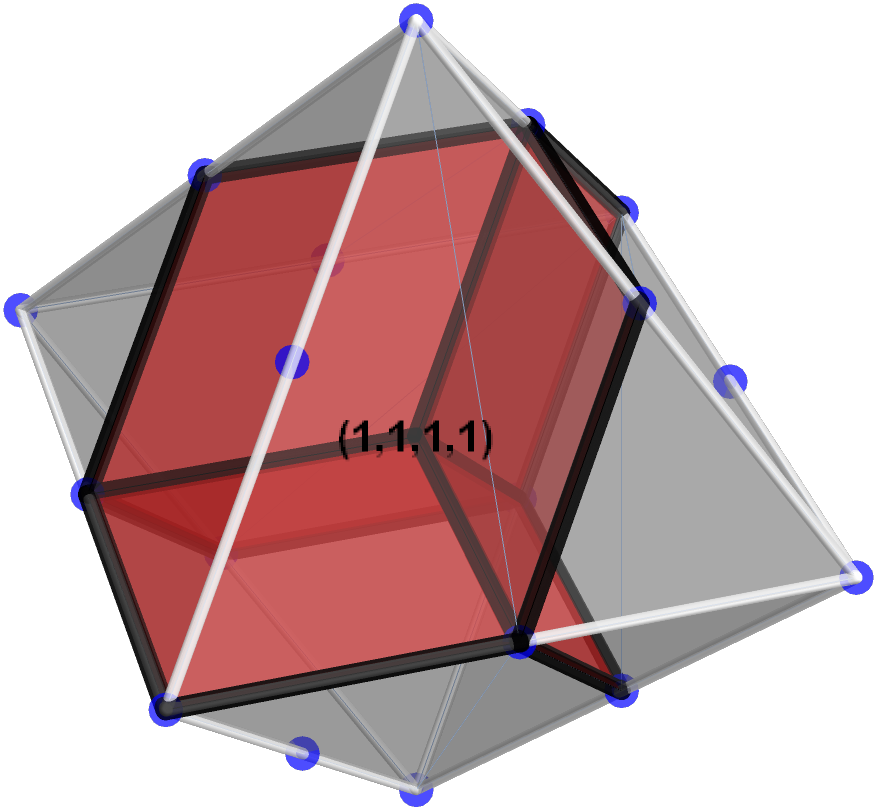}	
	\caption{The blade $((1_1 2_1 3_1 4_1))$ induces a 4-split.  From Example \ref{example: matroid subdivision 4 coords}: the blade $((1,2,3,4,5,6,7,8))_{e_{2468}}$ induces the same 4-split of $\Delta_{4,8}$ as the blade $((12_1 34_1 56_1 78_1))$.  Here $\Delta_{4,8}$ is viewed via its projection onto the second dilation of the octahedron $\Delta_{2,4}$ (right).
}
	\label{fig:4-9-2019-octahedron-in-blade}
\end{figure}

	\begin{figure}[h!]
	\centering
	\includegraphics[width=0.7\linewidth]{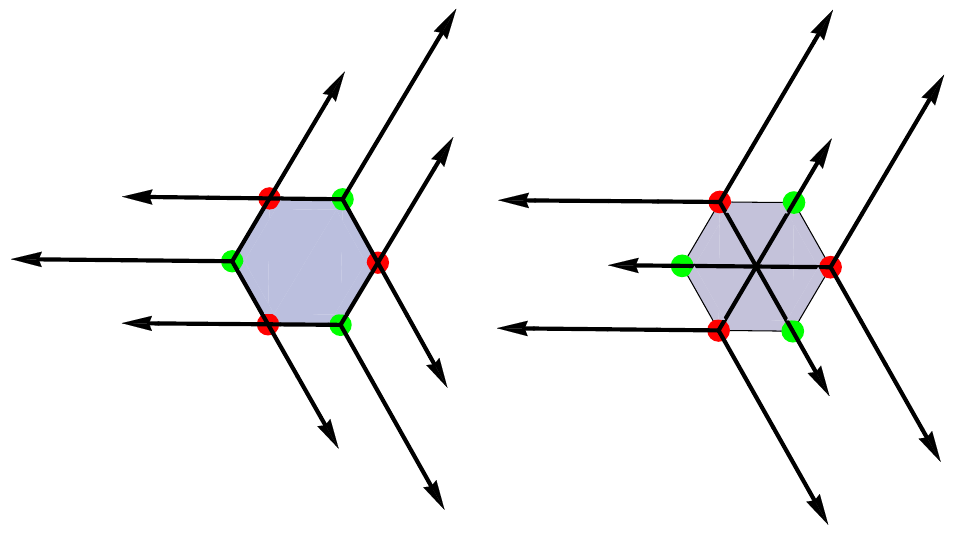}
	\caption{Two arrangements of the blade $((1,2,3))$ on the vertices of a hexagon; vertices  are permutations of $(0,1,2)$.  Blade arrangement on left (induces the trivial subdivision): $ \left\{((1_2 2_0 3_1)), ((1_0 2_1 3_2)), ((1_1 2_2 3_0))\right\}$.  Blade arrangement on right (induces a 6-chamber subdivision): $\left\{((1_2 2_1 3_0)), ((1_1 2_0 3_2)), ((1_0 2_2 3_1))\right\}$.  See Example \ref{example: hexagon subdivision} for details.
	}
	\label{fig:hexagon-blade-arrangement}
\end{figure}

\begin{example}\label{example: hexagon subdivision}
	Figure \ref{fig:hexagon-blade-arrangement} presents the projections of the two arrangements of blades on the vertices of $\Delta_{3,6}$ onto the vertices of a hexagon in the affine hyperplane in $\mathbb{R}^3$ where $y_{123}=3$, via 
	$$\mathbf{x}\mapsto (x_1+x_2,x_3+x_4,x_5+x_6) =: (y_1,y_2,y_3),$$
	where $y_1,y_2,y_3$ are the coordinate functions on $\mathbb{R}^3$.
	
	Then via this projection we have respectively
	$$\left\{((1_1 2_1 3_0 4_0 5_0 6_1)), ((1_0 2_0 3_0 4_1 5_1 6_1)), ((1_0 2_1 3_1 4_1 5_0 6_0))\right\} \mapsto  \left\{((1_2 2_0 3_1)), ((1_0 2_1 3_2)), ((1_1 2_2 3_0))\right\}$$
	and
	$$\left\{((1_1 2_1 3_0 4_1 5_0 6_0)), ((1_0 2_1 3_0 4_0 5_1 6_1)), ((1_0 2_0 3_1 4_1 5_0 6_1))\right\} \mapsto \left\{((1_2 2_1 3_0)), ((1_1 2_0 3_2)), ((1_0 2_2 3_1))\right\}.$$

\end{example}

\begin{example}
	Consider the tessellation in Figures \ref{fig:hexagonalbladetessellation}, \ref{fig:weightpermtessellation} and \ref{fig:weightpermtessellation2}.  Around each of the following vertices in Figure \ref{fig:weightpermtessellation} we have tripods oriented with respect to the 3-cycle $(1,2,3)$,
	$$\{((1,2,3))_{3e_1+e_2+2e_3},((1,2,3))_{6e_3},((1,2,3))_{4e_1-3e_2+5e_3}\},$$
	while the other three are oriented with respect to the 3-cycle $(1,3,2)$:
	$$\{((1,3,2))_{3e_1+3e_2},((1,3,2))_{4e_1+2e_3},((1,3,2))_{e_2+5e_3}\}.$$

	\begin{figure}[!htb]
		\minipage{0.5\textwidth}
		\includegraphics[width=.9\linewidth]{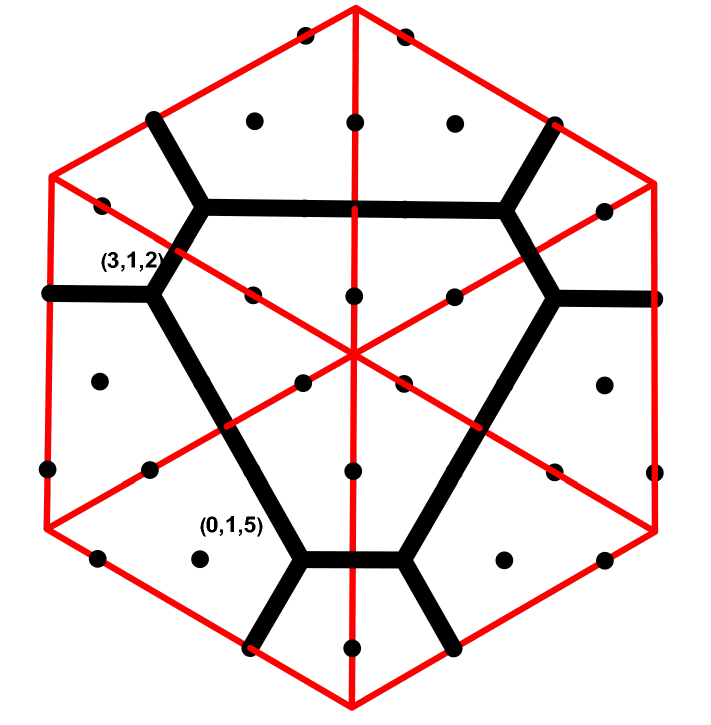}
		\caption{One choice of a (hexagonal) period; opposite edges of the red hexagon are identified. Red (medium thickness) edges are segments of affine hyperplanes, placed at $x_1-x_2 \in 3+6\mathbb{Z},\ x_2-x_3 \in 1+6\mathbb{Z},\ x_3-x_1\in 2+6\mathbb{Z}$.  Black edges (thickest) in the tiling are segments parallel to the three root directions $e_i-e_j$.  Near the vertices of the (three) weight permutohedra (with black edges of lengths respectively 1,2,3) the black shell coincides with a blade.}
		\label{fig:hexagonalbladetessellation}
		\endminipage\hfill
		\minipage{0.4\textwidth}
		\includegraphics[width=1.15\linewidth]{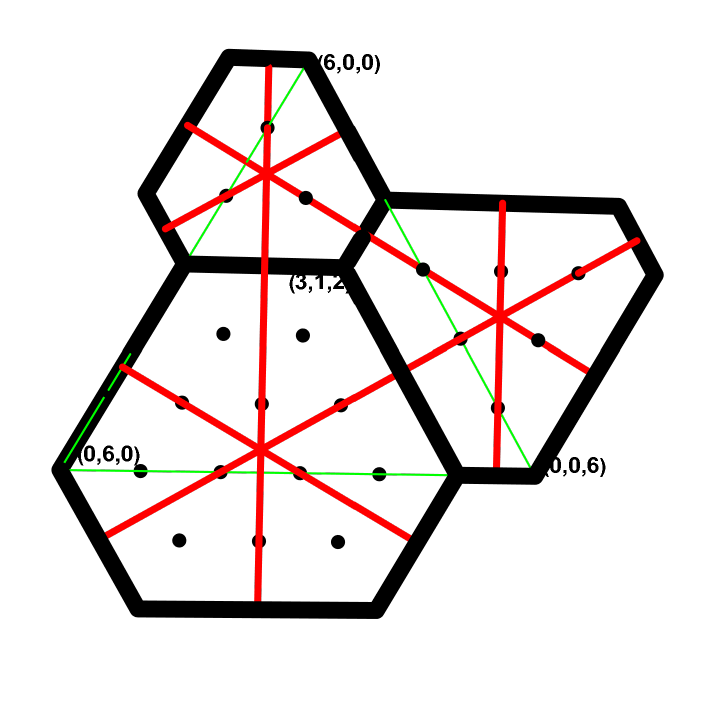}
\caption{Tessellating with three generalized (weight) permutohedra.  Here rather than taking rational-valued coordinates for the point $y$, we have dilated the Weyl alcoves by a factor 6; then the point $y=(3,1,2)$ is reflected across the (red, medium thickness) affine reflection hyperplanes placed at $x_1-x_2 \in 3+6\mathbb{Z},\ x_2-x_3 \in 1+6\mathbb{Z},\ x_3-x_1\in 2+6\mathbb{Z}$.
}		\label{fig:weightpermtessellation}

		\endminipage\hfill
	\end{figure}

\end{example}

	\begin{cor}\label{cor: arbitrary multi-split}
		For any multi-split matroid subdivision $\mathcal{S}$ of the hypersimplex $\Delta_{k,n}$, then there exists a vertex $e_J\in\Delta_{k,n}$ and a cyclic order $(\sigma_1,\ldots, \sigma_n)$ on the set $\{1,\ldots, n\}$, such that 
		$$\left(((\sigma_1,\ldots,\sigma_n))_{e_J}\right)\cap \Delta_{k,n}$$
		induces $\mathcal{S}$.
	\end{cor}
	\begin{proof}
		We claim first that $\mathcal{S}$ is induced by a unique blade $(((S_1)_{s_1},\ldots, (S_\ell)_{s_\ell}))$, where 
		$$((S_1)_{s_1},\ldots, (S_\ell)_{s_\ell})\in\text{OSP}(\Delta_{k,n}).$$  
		
		Indeed, as in Theorem \ref{thm: multisplit blade}, translating from Theorem 3.14 of \cite{SchroterMultisplits}, any (nontrivial) multi-split of the hypersimplex $\Delta_{k,n}$ has for its maximal cells the $\ell\ge 2$ cyclic block rotations of some nested matroid $\lbrack (S_1)_{s_1},\ldots, (S_\ell)_{s_\ell}\rbrack,$ where $( (S_1)_{s_1},\ldots, (S_\ell)_{s_\ell})\in\text{OSP}(\Delta_{k,n})$.

		Choose a cyclic order that is compatible with the ordered set partition $(S_1,\ldots, S_\ell)$ (for instance, put each block in increasing order and concatenate); then the formula from Theorem \ref{thm: multisplit blade} (keeping track of the reordering of $\sigma$) we obtain the integers $s_1,\ldots, s_\ell$ for the decorated permutation $((S_1)_{s_1},\ldots, (S_\ell)_{s_\ell})$, from which we can read off the vertex $e_J$ of the hypersimplex, and we find 
		$$(((S_1)_{s_1},\ldots, (S_\ell)_{s_\ell}))\cap \Delta_{k,n} = ((\sigma_1,\ldots,\sigma_n))_{e_{J}} \cap \Delta_{k,n}.$$
	\end{proof}

\section{Weakly separated arrangements induce matroid subdivisions}
	
\begin{figure}[h!]
	\centering
	\includegraphics[width=.55\linewidth]{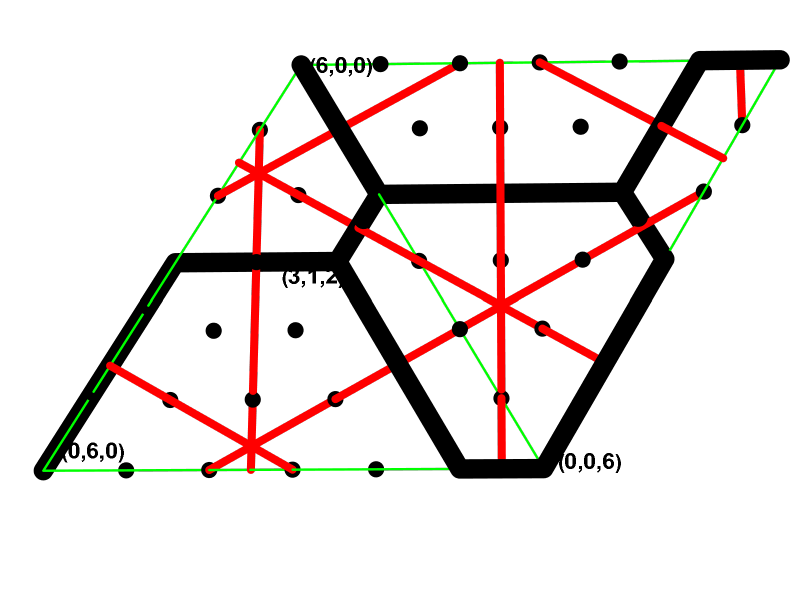}
	\caption{Same as in Figures \ref{fig:hexagonalbladetessellation} and \ref{fig:weightpermtessellation}, but shifted into the fundamental parallelepiped.  The lengths of the arms extending from the point $(3,1,2)$ add to the edge length $(=6)$ of the ambient simplex (outlined in green on bottom left).  Such subdivisions have appeared in various contexts, including \cite{PostnikovPermutohedra} and \cite{OcneanuVideo}.}
	\label{fig:weightpermtessellation2}
\end{figure}

We are interested in subsets of the set $\binom{\lbrack n\rbrack}{k}$ of $k$-element subsets of $\{1,\ldots, n\}$ which are weakly separated with respect to the natural cyclic order $(12\cdots n)$, in the sense of \cite{LeclercZelevinsky}.

\begin{defn}[\cite{LeclercZelevinsky}]\label{defn:weakly separated}
	Let $I,J\in\binom{\lbrack n\rbrack}{k}$ be given.  
	
	The subsets $I,J$ are \textit{weakly separated} if they satisfy the property that no four elements $i_1,i_2,j_1,j_2$ with $i_1, i_2\in (I\setminus J)$ and $j_1,j_2\in (J\setminus  I)$ have
	$$i_1<j_1<i_2<j_2$$
	or one of its cyclic rotations.

	If subsets $J_1,\ldots, J_m\in\binom{\lbrack n\rbrack}{k}$ are pairwise weakly separated, then $\mathcal{C} = \{J_1,\ldots, J_m\}$ is called a \textit{weakly separated collection}.
\end{defn}
In the usual geometric interpretation for $k$-element subsets, c.f. \cite{WeakSeparationPostnikov}, $I$ and $J$ are weakly separated if there exists a chord separating the sets $I\setminus J$ and $J\setminus I$ when drawn on a circle.  Identifying each $k$-element subset $J$ of $\{1,\ldots, n\}$ with the vertex $e_J$ gives rise to a notion of weak separation for arrangements of vertices of the form $\{e_{I_1},\ldots, e_{I_m}\}\subset \Delta_{k,n}$.

For any $m$-element subset $J$ of $\{1,\ldots, n\}$ where $1\le m\le k$ and any matroid polytope $\Pi\subseteq \Delta_{k,n}$, denote by $\partial_{(J,1)}(\Pi)$ the face of $\Pi$ where $x_j=1$ for all $j\in J$.  Similarly, if $I$ is an $m$-element subset with $1\le m\le n-k$, then denote by $\partial_{(I,0)}(\Pi)$ the face of $\Pi$ where $x_i=0$ for all $i\in I$.

Note that in Lemma \ref{lem: to the boundary hypersimplex}, the intersection may be completely uninteresting: it might not even induce a nontrivial subdivision!  Indeed, such degeneration is necessarily present, by the simple fact that a weakly separated collection of $k$-element subsets can have more elements than is possible for a weakly separated collection of $(k-1)$-element subsets.

For our present purposes, we may always discard such intersections when they appear -- see Example \ref{example: degenerating blade arrangements} where such behavior is illustrated.

\begin{lem}\label{lem: to the boundary hypersimplex}
	We have
	$$\partial_{(\{j\},1)}\left(((1,2,\ldots, n))_{e_{I}}\cap\Delta_{k,n}\right)= ((1,2,\ldots, \widehat{j},\ldots, n))_{e_{I'}}\cap \partial_{(\{j\},1)}(\Delta_{k,n}),$$
	where if $i_a < j \le i_{a+1}$ then $I' = I\setminus\{i_{a+1}\}$, and where the indices are cyclic as usual.  Moreover, if $\Pi \subseteq \Delta_{k,n}$ is any matroid polytope, then for any (nonempty) subset $J=\{j_1,\ldots, j_t\}\in \binom{n}{t}$ with $t\le k$, then the boundary 
	$$\partial_{(J,1)}(\Pi) := \partial_{(\{j_1\},1)}(\cdots (\partial_{(\{j_t\},1)}(\Pi))\cdots )$$
	is independent of the order of composition.  
\end{lem}

\begin{proof}
	The independence of the composition order is obvious geometrically.  We prove the formula for $I'$ when $J$ is a singleton.
	
	Let $((S_1)_{s_1},\ldots, (S_\ell)_{s_\ell}) \in \text{OSP}(\Delta_{k,n})$ be the decorated ordered set partition determined by the intersection with the hypersimplex,
	$$((1,2,\ldots, n))_{e_{I}}\cap \Delta_{k,n} = (((S_1)_{s_1},\ldots, (S_\ell)_{s_\ell}))\cap \Delta_{k,n}.$$
	Let us first suppose that $j$ is in some block $S_g$, with $s_g\ge 2$.  According to the construction of Theorem \ref{thm: multisplit blade}, the condition $s_g\ge 2$ means that the cyclic interval containing $j$ in the set $I$ has length at least two.  From Definition \ref{defn:blade} it follows immediately that we have, correspondingly,
	\begin{eqnarray*}
	& & \partial_{(\{j\},1)}\left((((S_1)_{s_1},\ldots,(S_g)_{s_g},\ldots,  (S_\ell)_{s_\ell})) \cap \Delta_{k,n}\right)\\
	 & = &  (((S_1)_{s_1},\ldots,(S_g\setminus \{j\})_{s_g-1},\ldots,  (S_\ell)_{s_\ell})) \cap \partial_{(\{j\},1)}(\Delta_{k,n})\\
	& = & ((1,2,\ldots,\widehat{j},\ldots, n))_{e_{I\setminus \{p\}}}\cap \partial_{(\{j\},1)}(\Delta_{k,n}),
	\end{eqnarray*}
	where $p=j$ if $j\in I$ and otherwise $p$ is the (cyclically) next element after $j$ in $I$.
	
	For clarity we shall justify the last line by working the two possible cases, respectively $j\in I$ and $j\not\in I$, for a particular blade arrangement.  With $n=8$ and $I = \{1,3,4,5,7\}$, we find
	\begin{eqnarray*}
		((1,2,\ldots, 8))_{e_I}\cap \Delta_{5,8} & = & ((18_1 2345_3 67_1)) \cap \Delta_{5,8},
	\end{eqnarray*}
	so that 
	\begin{eqnarray*}
		\partial_{(\{4\},1)}\left(((1,2,\ldots, 8))_{e_I}\cap \Delta_{5,8}\right) & = & \partial_{(\{4\},1)}\left(((18_1 2345_3 67_1)) \cap \Delta_{5,8}\right)\\
		& = & ((18_1 235_2 67_1)) \cap \partial_{(\{4\},1)}(\Delta_{5,8})\\
		& = & ((1,2,3,5,6,7,8))_{e_{1357}}\cap \partial_{(\{4\},1)}(\Delta_{5,8}).
	\end{eqnarray*}
	Here $j=4\in I$ and we replace $I$ with $I' = I\setminus\{4\}$.
	
	Similarly, 
		\begin{eqnarray*}
		\partial_{(\{2\},1)}\left(((1,2,\ldots, 8))_{e_I}\cap \Delta_{5,8}\right) & = & \partial_{(\{2\},1)}\left(((18_1 2345_3 67_1)) \cap \Delta_{5,8}\right)\\
		& = & ((18_1 345_2 67_1)) \cap \partial_{(\{2\},1)}(\Delta_{5,8})\\
		& = & ((1,3,4,5,6,7,8))_{e_{1457}}\cap \partial_{(\{2\},1)}(\Delta_{5,8}).
	\end{eqnarray*}
	Here $j=2\not\in I$ and we replace $I$ with $I' = I\setminus\{3\}$.
	
	Supposing now that $j\in S_g$ with $s_g = 1$, then setting $x_j=1$ gives for the inequalities defining the blade,
	$$x_{S_g} \ge s_g=1 \Leftrightarrow x_{(S_g\setminus \{j\})} \ge 0,$$
	which is a trivial consequence that we are in $x\in\Delta_{k,n}$ (in particular all coordinates of $x$ are already nonnegative).  This means that $S_g$ becomes $S_g\setminus \{j\}$ and joins the cyclically next block:
	\begin{eqnarray*}
	((S_1)_{s_1},\ldots, (S_g)_{s_g=1},\ldots, (S_\ell)_{s_\ell}) & \mapsto & ((S_1)_{s_1},\ldots, ((S_g \setminus\{j\})\cup S_{g+1})_{s_{g+1}},\ldots, (S_\ell)_{s_\ell})\\
	& = & ((1,2,\ldots,\widehat{j},\ldots, n))_{e_{I\setminus \{p\}}},
\end{eqnarray*}
	where if $j\in I$ then $I' = I\setminus \{j\}$, and otherwise $I' = I\setminus \{p\}$ where $p\in I$ comes cyclically after $j$.
	
\end{proof}

\begin{example}\label{example: degenerating blade arrangements}
	Consider the blade $((1,2,3,4,5,6))$ arranged on the collection of vertices $$\{e_{124},e_{246},e_{256},e_{346}\}$$ of $\Delta_{3,6}$.  These coincide on $\Delta_{3,6}$ with the blades respectively 
	$$((1256_2 34_1)),\ ((12_1 34_1 56_1)),\ ((12_1 3456_2)),\ ((1234_2 56_1)).$$
	On the boundary where $x_3=1$, say, we find 
	$$((1256_2)),\ ((12_1 456_1)),\ ((12_1 456_1)),\ ((124_1 56_1)).$$
	Discarding the trivial subdivision and deleting duplicates we obtain a pair of compatible splits
	$$((12_1 456_1))\cap \partial_{(\{3\},1)}(\Delta_{3,6}) = ((1,2,4,5,6))_{e_{26}}\cap \partial_{(\{3\},1)}(\Delta_{3,6})$$
	and
	$$((124_1 56_1))\cap \partial_{(\{3\},1)}(\Delta_{3,6}) = ((1,2,4,5,6))_{e_{46}}\cap \partial_{(\{3\},1)}(\Delta_{3,6}).$$
	In this case this corresponds to the maximal weakly separated collection of 2-element subsets of $\{1,2,4,5,6\}$,
	$$\{\{2,6\},\{4,6\}\}.$$
	Using the rule from Lemma \ref{lem: to the boundary hypersimplex}, setting $x_3 \rightarrow 1$, the vertex arrangement changes as follows:
	$$e_{124}\mapsto e_{12},\ e_{246}\mapsto e_{26},\ e_{256} \mapsto e_{26},\ e_{346}\mapsto e_{46},$$
	and after removing frozen vertices and removing duplicates we recover the pair $\{e_{26},e_{46}\}$ as above.
	
\end{example}

In Corollary \ref{cor: subdivision arrangement boundary} we first review the construction of the decorated ordered set partition from a vertex $e_I\in \Delta_{k,n}$ from Theorem \ref{thm: multisplit blade}, and then state the matroid subdivision on each boundary component of $\Delta_{k,n}$ where $x_j=1$.

\begin{cor}\label{cor: subdivision arrangement boundary}
	Choose any vertex $e_I\in\Delta_{k,n}$ with cyclic invervals, say,
	$$I=\{i_1,i_1+1,\ldots, i_1+({\lambda_1}-1)\} \cup \{i_2,i_2+1,\ldots, i_2+({\lambda_2}-1)\}\cup \cdots\cup \{i_{\ell},i_{\ell}+1,\ldots, i_\ell + (\lambda_{\ell}-1)\},$$
	so that we have $\lambda_1+\cdots+ \lambda_\ell = k$.  Set $I_j = \{i_j,i_j+1,\ldots, i_j+({\lambda_j}-1)\} $.  Let $(C_1,\ldots, C_\ell)$ be the interlaced complement to the intervals $I$, so that we have the concatenation
	$$(1,2,\ldots, n) = (C_1,I_1,C_2,I_2,\ldots, C_\ell,I_\ell).$$
	By Theorem \ref{thm: multisplit blade} we have 
	$$\left(((1,2,\ldots, n))_{e_{I}}\right)\cap \Delta_{k,n} = \left(((S_1)_{s_1},(S_2)_{s_2},\ldots, (S_\ell)_{s_\ell})\right)\cap \Delta_{k,n},$$
	where $(\mathbf{S}_{\mathbf{s}}) = ((S_1)_{s_1},\ldots, (S_\ell)_{s_\ell}) \in \text{OSP}(\Delta_{k,n})$ is the decorated ordered set partition defined by $(S_j,s_j) = (I_j\cup C_j , \vert I_j \vert )$. 
	
	By Lemma \ref{lem: to the boundary hypersimplex}, for each $j=1,\ldots, n$, we have	
	$$\partial_{(\{j\},1)}\left(((1,2,\ldots, n))_{e_I}\cap \Delta_{k,n}\right) = ((1,2,\ldots, \widehat{j},\ldots, n))_{e_{I'}}\cap \partial_{(\{j\},1)}(\Delta_{k,n}),$$
	where $I'$ is given by Lemma \ref{lem: to the boundary hypersimplex}.  
	
		\vspace{.5cm}
	Then, on the boundary $\partial_{(\{j\},1)}(\Delta_{k,n})$ we have the matroid subdivision with maximal cells the $\ell'$ cyclic block rotations of 
	$$\lbrack (T_1)_{t_1},(T_2)_{t_2},\ldots, (T_{\ell'})_{t_{\ell'}}\rbrack$$
	where $((T_1)_{t_1},(T_2)_{t_2},\ldots, (T_{\ell'})_{t_{\ell'}})\in\text{OSP}\left(\partial_{(\{j\},1)}(\Delta)\right)$ is defined as follows.
	\begin{enumerate}
		\item If $j\in S_a$, say,  and $s_a=1$ then replace $((S_a)_{1},(S_{a+1})_{s_{a+1}})$ with $((S_a \setminus \{j\}) \cup (S_{a+1}))_{s_{a+1}}$.  In this case $\ell' = \ell-1$.
		\item If $j\in S_a$ and $s_a\ge 2$ then replace $(S_a)_{s_a}$ with $((S_a\setminus\{j\})_{s_a-1})$.  In this case $\ell'=\ell$.
	\end{enumerate}
	Now we concatenate the resulting decorated blocks and obtain $\lbrack (T_1)_{t_1},(T_2)_{t_2},\ldots, (T_{\ell'})_{t_{\ell'}}\rbrack$.
\end{cor}

\begin{rem}
	As of this writing, we have not yet considered the natural general question, in which cases (or if always) one obtains, by restricting to a facet $x_j=1$ a blade arrangement which corresponds to a maximal weakly separated arrangement of $k$-element subsets of an $n$-element set, a blade arrangement which corresponds to a maximal weakly separated collection of $(k-1)$-element subsets of an $(n-1)$ element set.  Here we have adopt the convention that we do not count the ``frozen'' sets $\{i,i+1,\ldots, j\}$ as possible elements in a weakly separated collection.  Based in part on experiment we expect the answer to be that a maximal weakly separated collection of $k$-element subsets maps to a maximal weakly separated collection of $(k-1)$-element subsets, but the proof was beyond the scope of this paper and is left to future work.
	
	However we \textit{have} checked explicitly in Mathematica that maximal weakly separated collections for $\Delta_{3,6}$ and $\Delta_{3,7}$ map via Lemma \ref{lem: to the boundary hypersimplex} to respectively maximal weakly separated collections for $\Delta_{2,5}$ and $\Delta_{2,6}$, where as usual we have to disregard frozen sets and eliminate redundancy.  For instance, for $\Delta_{3,7}$, the 259 maximal weakly separated collections of six 3-element subsets of $\{1,\ldots, 7\}$ map (with varying multiplicity) to the 14 maximal weakly separated collections of three 2-element subsets of $\{1,\ldots, \widehat{j},\ldots, 7\}$ for each $j=1,\ldots, 7$.
\end{rem}

\begin{rem}\label{rem:subsetweaklyseparable}
	It is easy to see that if $\{I,J\} \subset \binom{\lbrack n\rbrack}{k}$ is a weakly separated pair, and we have $\{I',J'\} \subset \binom{\lbrack n\rbrack}{k-1}$ with $I'\subset I$ and $J'\subset J$, then $\{I',J'\}$ is a weakly separated pair.
\end{rem}

Basic considerations in matroid theory imply Proposition \ref{prop: secondhypersimplex test}; nonetheless we sketch a proof.

\begin{prop}\label{prop: secondhypersimplex test}
	A polytope $\Pi \subset \Delta_{k,n}$ with $k\ge 3$ is a matroid polytope if and only if every (nonempty) face $\partial_{(\{j\},1)} (\Pi) \subseteq \partial_{(\{j\},1)} (\Delta_{k,n}) \simeq \Delta_{k-1,n-1}$ is a matroid polytope, for all facets $x_j=1$, as $j=1,\ldots, n$.
\end{prop}

\begin{proof}[Sketch of proof]
	Translating from \cite{Maurer}, the basis exchange relations may be checked on each octahedral face of the candidate matroid polytope $\Pi\subseteq\Delta_{k,n}$.  In particular, for $k\ge3$, then every octahedral face of $\Pi$ is also in some facet $\partial_{(\{j\},1)}(\Pi)$.  The result follows.
\end{proof}

Combinatorially speaking, our second main result, Theorem \ref{thm: weakly separated matroid subdivision blade}, shows that blades provide the vehicle to formulate weak separation for $k$-element subsets in terms of matroid subdivisions.  Theorem \ref{thm: weakly separated matroid subdivision blade} also generalizes the notion of compatibility for splits to multi-split matroid subdivisions of any hypersimplex.

Once we prove the equivalence for weakly separated pairs in Lemma \ref{lem: weakly separated matroid subdivision blade} the equivalence for $m$-element collections is fairly straightforward to deduce.
\begin{lem}\label{lem: weakly separated matroid subdivision blade}
	Given a pair of vertices $e_{I_1},e_{I_2}\in\Delta_{k,n}$, the blade arrangement 
	$$\{((1,2,\ldots, n))_{e_{I_1}},((1,2,\ldots, n))_{e_{I_2}}\}$$
	induces a matroid subdivision of $\Delta_{k,n}$ if and only if $\{I_1,I_2\}$ is weakly separated.	
\end{lem}

\begin{proof}
	Our plan is to induct on $k\ge 2$.
	
	So let $k=2$ and $n\ge 4$.  Here the equivalence between weak separation and compatibility of 2-splits of $\Delta_{2,n}$ is obvious, particularly so if the intervals are drawn on a circle.  Nonetheless, for concreteness we include the proof.
	
	Each pair $\mathcal{C}=\{\{i_1,i_2\}, \{j_1,j_2\}\}$ determines four cyclic intervals,
	\begin{eqnarray*}
		U & = & (i_1+1,\ldots, i_2),\ U^c = (i_2+1,\ldots, i_1)\\
		V & = & (j_1+1,\ldots, j_2),\ U^c = (j_2+1,\ldots, j_1).
	\end{eqnarray*}
	If the pair $\mathcal{C}$ is \textit{not} weakly separated, then we have the cyclic order $i_1<j_1<i_2<j_2$, so $i_2\in U\cap V$, $j_1\in U\cap V^c$, $i_2\in U^c\cap V$ and $j_2\in U^c\cap V^c$, so all four intersections 
	\begin{eqnarray*}\label{eq: four intersections}
		U\cap V, & \  &U\cap V^c\\
		U^c\cap V,&\ & U^c\cap V^c
	\end{eqnarray*}
	are nonempty.  Conversely, if all four intersections are nonempty, we can recover the pairs $\{i_1,i_2\}, \{j_1,j_2\}$ by taking the cyclically last elements in each interval $U,U^c, V,V^c$.  Consequently, we have shown that $\mathcal{C}$ is weakly separated if and only if at least one of the four intersections of Equation \eqref{eq: four intersections} is empty; by Corollary \ref{cor:blade splits} this is exactly the compatibility condition for 2-splits of $\Delta_{2,n}$ from \cite{Hiroshi}.  This, together with Theorem \ref{thm: multisplit blade}, shows that 
	$$((1,2,\ldots, n))_{e_{i_1,i_2}}, ((1,2,\ldots, n))_{e_{j_1,j_2}}$$
	is matroidal if and only if $\{\{i_1,i_2\},\{j_1,j_2\}\}$ is weakly separated.  This completes the base step of the induction.
	
	Now take $k\ge 3$ with $n\ge k+2$. Suppose first that $\{I_1,I_2\}\subset \binom{\lbrack n\rbrack}{k}$ is weakly separated.
	
	Let $(\mathbf{S})_{\mathbf{s}} = ((S_1)_{s_1},\ldots, (S_\ell)_{s_\ell}),\  (\mathbf{U})_{\mathbf{u}}  = ((U_1)_{u_1},\ldots, (U_\ell)_{u_{\ell'}}) \in \text{OSP}(\Delta_{k,n})$, as constructed in the proof of Theorem \ref{thm: multisplit blade}, satisfying
	$$((1,2,\ldots, n))_{e_{I_1}} \cap \Delta_{k,n} = (((S_1)_{s_1},\ldots, (S_\ell)_{s_\ell}))\cap \Delta_{k,n}$$
	and
	$$((1,2,\ldots, n))_{e_{I_2}} \cap \Delta_{k,n} = (((U_1)_{u_1},\ldots, (U_{\ell'})_{u_{\ell'}}))\cap \Delta_{k,n}.$$
	Let $L = \{1,\ldots, \ell\}$ and $L' = \{1,\ldots, \ell'\}$ and let 
	$$K = \{(i,j)\in L\times L': \dim\left(\lbrack (S_i)_{s_i},\ldots, (S_{i-1})_{s_{i-1}}\rbrack\cap \lbrack ((U_j)_{u_j},\ldots, (U_{j-1})_{u_{j-1}})\rbrack\right) = n-1 \}$$
	be the indexing set for the maximal cells of the subdivision induced by refining the matroid subdivisions induced by respectively
	$$(((S_1)_{s_1},\ldots, (S_\ell)_{s_\ell})) \text{ and }  (((U_1)_{u_1},\ldots, (U_{\ell'})_{u_{\ell'}})).$$
	We claim that for each facet $\partial_{(\{j\},1)}(\Delta_{k,n})$, the polytopes 
	$$\{\lbrack ((S_a)_{s_a},\ldots, (S_{a-1})_{s_{a-1}})\rbrack \cap \lbrack ((U_b)_{u_b},\ldots, (U_{b-1})_{u_{b-1}})\rbrack: (a,b)\in K\}$$
	are the maximal cells of the subdivisions of $\Delta_{k,n}$ induced by $((1,2,\ldots, n))_{e_{I_1}},((1,2,\ldots, n))_{e_{I_2}}$	and that, in particular, these intersections are matroid polytopes.
	
	Now each blade induces separately a matroid subdivision of $\partial_{(\{j\},1)}(\Delta_{k,n})$, and by Lemma \ref{lem: to the boundary hypersimplex} we have
\begin{eqnarray}\label{eq:intersection1}
\partial_{(\{j\},1)}\left( ((1,2,\ldots, n))_{e_{I_1}}\cap \Delta_{k,n}\right) = ((1,2,\ldots, \widehat{j},\ldots, n))_{e_{I'_1}}\cap \partial_{(\{j\},1)}(\Delta_{k,n})
\end{eqnarray}
	and
\begin{eqnarray}\label{eq:intersection2}
\partial_{(\{j\},1)}\left( ((1,2,\ldots, n))_{e_{I_2}}\cap \Delta_{k,n}\right) = ((1,2,\ldots, \widehat{j},\ldots, n))_{e_{I'_2}}\cap \partial_{(\{j\},1)}(\Delta_{k,n}).
\end{eqnarray}

	Replacing $((T_1)_{t_1},(T_2)_{t_2},\ldots, (T_{\ell'})_{t_{\ell'}})$ in Corollary \ref{cor: subdivision arrangement boundary} with respectively 
		$$\lbrack (S_a)_{s_a},\ldots, (S_{a-1})_{s_{a-1}}\rbrack \text{ and } \lbrack (U_b)_{u_b},\ldots, (U_{b-1})_{u_{b-1}}\rbrack$$
	it follows that the maximal cells induced by the arrangements
	$$((1,2,\ldots, \widehat{j},\ldots, n))_{e_{I'_1}}\text{ and }((1,2,\ldots, \widehat{j},\ldots, n))_{e_{I'_2}}.$$
	are respectively
	$$\partial_{(\{j\},1)}\left(\lbrack (S_a)_{s_a},\ldots, (S_{a-1})_{s_{a-1}}\rbrack\right) \text{ and } \partial_{(\{j\},1)}\left(\lbrack (U_b)_{u_b},\ldots, (U_{b-1})_{u_{b-1}}\rbrack\right).$$
		But recalling that we have assumed $\{I_1,I_2\}$ to be weakly separated, then $\{I'_1,I'_2\}$ is weakly separated and it follows inductively that the blade arrangement 
	$$\{((1,2,\ldots, \widehat{j},\ldots, n))_{e_{I'_1}},((1,2,\ldots, \widehat{j},\ldots, n))_{e_{I'_2}}\}$$
	is matroidal, and that its maximal (matroid) cells are exactly the intersections
	$$ \partial_{(\{j\},1)}\left(\lbrack ((S_a)_{s_a},\ldots, (S_{a-1})_{s_{a-1}})\rbrack\right) \cap \partial_{(\{j\},1)}\left(\lbrack ((U_b)_{u_b},\ldots, (U_{b-1})_{u_{b-1}})\rbrack\right).$$	
	In view of the identity
	\begin{eqnarray*}
		&& \partial_{(\{j\},1)}\left(\lbrack ((S_a)_{s_a},\ldots, (S_{a-1})_{s_{a-1}})\rbrack\right) \cap \partial_{(\{j\},1)}\left(\lbrack ((U_b)_{u_b},\ldots, (U_{b-1})_{u_{b-1}})\rbrack\right)\\
		& = &  \partial_{(\{j\},1)}\left(\lbrack ((S_a)_{s_a},\ldots, (S_{a-1})_{s_{a-1}})\rbrack \cap \lbrack ((U_b)_{u_b},\ldots, (U_{b-1})_{u_{b-1}})\rbrack\right),
	\end{eqnarray*}
	repeating the argument for each $j=1,\ldots, n$, i.e. for all facets, using Proposition \ref{prop: secondhypersimplex test}, it now follows that the elements in the set 
		$$\{\lbrack ((S_a)_{s_a},\ldots, (S_{a-1})_{s_{a-1}})\rbrack \cap \lbrack ((U_b)_{u_b},\ldots, (U_{b-1})_{u_{b-1}})\rbrack: (a,b)\in K\}$$
	are matroid polytopes.
	
	\vspace{.3cm}
	
Conversely, if $\Pi_1,\ldots, \Pi_t$ are the maximal cells of the matroid subdivision of $\Delta_{k,n}$ that is induced by a given matroidal arrangement of the blade $((1,2,\ldots, n))$ on vertices $e_{I_1},e_{I_2}$, it follows (recall, we are assuming that $k\ge 3$) that every facet where $x_j=1$ for $j=1,\ldots, n$ of every $\Pi_i$ must also be a matroid polytope. 

Let us now assume to the contrary that $\{I_1,I_2\}$ is not weakly separated, i.e. we have
$$g_1,g_2\in I_1\ \text{ and } h_1,h_2 \in I_2$$
with, cyclically, $g_1<h_1<g_2<h_2$.

Now up to cyclic rotation, $I_1$ and $I_2$ have the form
$$I_1 = \{i_1,\ldots, g_1,\ldots, g_2,\ldots, i_k\},\ I_2 = \{j_1,\ldots, h_1,\ldots, h_2,\ldots, j_k\}.$$
 Choosing $j \in \{1,\ldots, n\}$ such that its cyclic successors in $I_1$ and $I_2$ are not, respectively, in $\{g_1,g_2\}$ and $\{h_1,h_2\}$, then applying the formula from Lemma \ref{lem: to the boundary hypersimplex}, we find that the pair $\{I_1', I_2'\}$ is still not weakly separated, so by induction the blade arrangement
$$\{((1,2,\ldots, \widehat{j},\ldots, n))e_{I'_{1}},((1,2,\ldots, \widehat{j},\ldots, n))e_{I'_{2}}\}$$
on $\partial_{(\{j\},1)}(\Delta_{k,n})\simeq\Delta_{k-1,n-1}$ is not matroidal, hence one of its maximal cells is not a matroid polytope.  But this is a contradiction, since the maximal cells in its induced subdivision are among the matroid polytopes 
$$\{\partial_{(\{j\},1)}(\Pi_1),\ldots, \partial_{(\{j\},1)}(\Pi_t)\}.$$

\end{proof}

Recall that the maximal cells of an arbitrary $\ell$-split matroid subdivision of $\Delta_{k,n}$ are $\ell$ nested matroid polytopes
$$\lbrack (S_1)_{s_1},(S_2)_{s_2},\ldots, (S_{\ell})_{s_\ell})\rbrack, \lbrack (S_2)_{s_2},(S_3)_{s_3}\ldots, (S_{1})_{s_1})\rbrack,\ldots, \lbrack (S_\ell)_{s_\ell},(S_{1})_{s_{1}}\ldots, (S_{\ell-1})_{s_{\ell-1}})\rbrack$$
for a decorated ordered set partition $((S_1)_{s_1},(S_2)_{s_2},\ldots, (S_{\ell})_{s_\ell})) \in \text{OSP}(\Delta_{k,n})$ that is uniquely determined up to cyclic block rotation.
\begin{cor}\label{cor:compatibility multi-splits}
	Consider an arbitrary pair $\Pi_1,\Pi_2$ of multi-splits of $\Delta_{k,n}$ with maximal cells the cyclic block rotations of the nested matroid polytopes
	$$\lbrack (S_i)_{s_i},(S_{i+1})_{s_{i+1}},\ldots, (S_{i-1})_{s_{i-1}}\rbrack,\ \ \lbrack (T_j)_{t_j},(T_{j+1})_{t_{j+1}},\ldots, (T_{j-1})_{t_{j-1}}\rbrack$$
	respectively, where 
	$$((S_1)_{s_1},\ldots, (S_{\ell})_{s_\ell}), ((S_1)_{s_1},\ldots, (S_{m})_{s_m}) \in \text{OSP}(\Delta_{k,n})$$
	are decorated ordered set partitions of hypersimplicial type $\Delta_{k,n}$.
	
	Suppose that there exists a cyclic order $\sigma = (\sigma_1,\ldots, \sigma_n)$, together with vertices $e_{J_1},e_{J_2}\in \Delta_{k,n}$ such that 
	$$(((S_1)_{s_1},\ldots, (S_{\ell})_{s_\ell})) \cap \Delta_{k,n} = ((\sigma_1,\ldots, \sigma_n))_{e_{J_1}} \cap \Delta_{k,n}$$
	and 
	$$(((T_1)_{t_1},\ldots, (T_{m})_{t_m})) \cap \Delta_{k,n} = ((\sigma_1,\ldots, \sigma_n))_{e_{J_2}} \cap \Delta_{k,n}.$$
	Then, the common refinement of $\Pi_1$ and $\Pi_2$ is matroidal if and only if the pair of vertices $e_{J_1},e_{J_2} \in\Delta_{k,n}$ is weakly separated with respect to the cyclic order $(\sigma_1,\ldots, \sigma_n)$.
	
\end{cor}

\begin{proof}
	This follows immediately from a translation of Lemma \ref{lem: weakly separated matroid subdivision blade}.
	
	Indeed, suppose that $\Pi_1,\Pi_2$ are multi-split matroid subdivisions induced by the blades
	$$((1,2,\ldots, n))_{e_{J_1}}\text{ and } ((1,2,\ldots, n))_{e_{J_2}},$$
	respectively, where without loss of generality we assume that the cyclic order is the standard one $\sigma = (1,2,\ldots, n)$, with $e_{J_1},e_{J_2}\in \Delta_{k,n}$.
	
	If the pair $\{e_{J_1},e_{J_2}\}$ is weakly separated, then Lemma \ref{lem: weakly separated matroid subdivision blade} implies that the arrangement 
	$$\left\{((1,2,\ldots, n))_{e_{J_1}}, ((1,2,\ldots, n))_{e_{J_2}}\right\}$$
	is matroidal, hence the common refinement of the two multi-splits $\Pi_1,\Pi_2$ is matroidal.  Conversely, if the blade arrangement is matroidal then it induces a subdivision that is matroidal; but this subdivision is the common refinement of $\Pi_1,\Pi_2$.
\end{proof}

\begin{rem}
	One can check that the arrangement of blades $\{((12_1 34_1 56_1)), ((12_1 56_1 34_1))\}$ on $\Delta_{3,6}$ is matroidal, that is, together they induce a subdivision of $\Delta_{3,6}$ that is positroidal; however this affirmative result would not follow from Corollary \ref{cor:compatibility multi-splits} because there is no single cyclic order compatible with both decorated ordered set partitions $(12_1 34_1 56_1)$ and $(12_1 56_1 34_1)$.  However, the analogous property does not hold in general for $k\ge 4$: the superposition of the blades $$\{((12_1 34_1 56_1 78_1)), ((12_1 78_1 56_1 34_1))\}$$ on $\Delta_{4,8}$ does not induce a matroid subdivision. 
\end{rem}

\begin{thm}\label{thm: weakly separated matroid subdivision blade}
	Given a collection of vertices $e_{I_1},e_{I_2},\ldots, e_{I_m}\in\Delta_{k,n}$, the blade arrangement 
$$\{((1,2,\ldots, n))_{e_{I_1}},((1,2,\ldots, n))_{e_{I_2}},\ldots, ((1,2,\ldots, n))_{e_{I_m}}\}$$
induces a matroid subdivision of $\Delta_{k,n}$ if and only if $\{I_1,\ldots, I_m\}$ is weakly separated.	
\end{thm}

\begin{proof}
	Let $\{I_1,\ldots, I_m\}\subset \binom{\lbrack n \rbrack}{k}$  be arbitrary.  The only remaining part of the proof here is to show that if each arrangement
	$$\{((1,2,\ldots,n))_{e_{I_p}},((1,2,\ldots,n))_{e_{I_q}}\}$$
	is matroidal for each (distinct) $p,q$, then the whole arrangement on the vertices $e_{I_1},\ldots, e_{I_m}$ is matroidal.
	
	Let us now suppose that $\{I_1,\ldots, I_m\}$ is weakly separated.
	
	Choose any octahedral face 
	$$\mathcal{F}_{(A,B),n} = \partial_{(A,1)}\left(\partial_{(B,0)}(\Delta_{k,n})\right) = \left\{x\in \Delta_{k,n}: x_a= 1\text{ and }x_b=0 \text{ for all } a\in A,b\in B\right\}$$
	of the hypersimplex $\Delta_{k,n}$, where $A\in \binom{\lbrack n\rbrack}{k-2}$ and $B\in\binom{\lbrack n\rbrack}{n-4-(k-2)}$ are disjoint.  Now it follows from Lemma \ref{lem: weakly separated matroid subdivision blade} that for each (distinct) $p,q\in \{1,\ldots, m\}$, the blade arrangement 
	$$\{((1,2,\ldots,n))_{e_{I_p}},((1,2,\ldots,n))_{e_{I_q}}\}$$
	is matroidal, so it induces either the trivial subdivision, or a (planar) 2-split matroid subdivision $\Sigma_{(A,B)}$ on $\mathcal{F}_{(A,B),n}$.  Consequently each blade $((1,2,\ldots, n))_{e_{I_p}}$ for $p=1,\ldots, m$ induces on $\mathcal{F}_{(A,B),n}$ either the trivial subdivision or $\Sigma_{(A,B)}$.  
	
	Supposing to the contrary that the whole blade arrangement 
	$$\{((1,2,\ldots, n))_{e_{I_1}},((1,2,\ldots, n))_{e_{I_2}},\ldots, ((1,2,\ldots, n))_{e_{I_m}}\}$$
	were not matroidal, then some octahedral face would have to be subdivided into the full alcove triangulation; that is it would be split in two different (nontrivial) ways by some pair 
	$$\{((1,2,\ldots, n))_{e_{I_{p'}}},((1,2,\ldots, n))_{e_{I_{q'}}}\},$$
	which is a contradiction.
\end{proof}

\begin{example}
	Consider the non-weakly separated pair of vertices $e_{124},e_{135}$ of $\Delta_{3,6}$.  Then we may choose either $j=1$ or $j=6$ for the boundary, in the proof of Theorem \ref{thm: weakly separated matroid subdivision blade}, and end up with the same non-weakly separated pair of vertices $\{e_{35},e_{24}\}$ of $\Delta_{2,5}$.
\end{example}

\begin{rem}
	In light of Theorem \ref{thm: weakly separated matroid subdivision blade}, for matroidal blade arrangements, the $n$ ``frozen'' vertices which have cyclically adjacent indices, do not change the subdivision of $\Delta_{k,n}$.  It follows from Theorem \ref{thm: weakly separated matroid subdivision blade} together with the purity conjecture (for $k$-element subsets), proven independently in \cite{Danilov,Danilov2,WeakSeparationPostnikov} that all maximal by inclusion weakly separated collections of vertices of $\Delta_{k,n}$ have the same cardinality $(k-1)(n-k-1)+n$, so that for general $\Delta_{k,n}$, each maximal (by refinement, and by size) matroidal blade arrangement consists of $(k-1)(n-k-1)$ copies of the blade $((1,2,\ldots, n))$ arranged on a weakly separated collection of vertices of $\Delta_{k,n}$.
\end{rem}

In the final two examples we aim to illustrate the nontriviality of our criterion in Corollary \ref{cor:compatibility multi-splits} even for 2-splits; more examples follow in Section \ref{sec: tree arrangements}.

\begin{example} 
	For $n=7$, from the (non-)weakly separated pair
	$$\mathcal{W} = \{e_{124},e_{135}\}$$
	we obtain the blade arrangement
	$$\{((12567_2 34_1)),((167_1 23_1 45_1))\}.$$
	It is not matroidal, since one of the maximal cells, the alcoved subpolytope of $\Delta_{3,7}$
	$$\Pi=\{x\in \Delta_{3,7}:x_{12567}\ge 2,\ x_{167}\ge 1,\ x_{16723}\ge 2\}$$
	is not a matroid polytope.  This could be seen by explicit computation of the convex hull using a software package: the edge connecting the vertices $e_{124}, e_{135}$ is an edge of $\Pi$ and it is in the non root direction 
	$$e_{124}- e_{135}=e_{24} - e_{35}.$$
	Equivalently, more practically, one can check explicitly that the matroid basis exchange relations fail (for instance) for the pair of bases $\{1,2,4\},\{1,3,5\}$ of the (candidate) matroid corresponding to $\Pi$.
\end{example}

\begin{example}
	The weakly separated pair $\{e_{125},e_{134}\}$ induces a matroidal blade arrangement:
	$$((126_2 345_1)),((156_1 234_2)).$$
	These induce compatible splits.  On the other hand, a blade arranged on the non-weakly separated arrangement of vertices $\{e_{124},e_{235}\}$, induces a subdivision which is not matroidal:
	$$\{((1256_2 34_1)),((1236_2 45_1))\}$$
	The difference between these two pairs of splits is not obvious from their equations alone, though the criterion is known for 2-splits of any hypersimplex \cite{HermannJoswig}, but here the criterion (for planar matroid subdivisions), that the vertices of a matroidal blade arrangement define a weakly separated collection, is purely combinatorial and as we have seen in Corollary \ref{cor:compatibility multi-splits} extends naturally to multisplits.
		
\end{example}
\section{From matroid subdivisions to their boundaries: tree arrangements}\label{sec: tree arrangements}

In this section we present an extended example, in which we work out in detail how Lemma \ref{lem: to the boundary hypersimplex} may be used to pass from (1) a maximal weakly separated collection of 3-element subsets of vertices of $\Delta_{3,7}$ through (2) a maximal matroidal blade arrangement, to obtain (3) an arrangement of $7$ trees which are dual to the respective matroid subdivisions induced on the $n$ boundary copies of $\Delta_{2,6}$.  

However, we remark that using Lemma \ref{lem: to the boundary hypersimplex} one can pass directly from a weakly separated collection of vertices on $\Delta_{k,n}$ to a weakly separated collection of vertices on $\Delta_{k-1,n-1}$.
\begin{figure}[h!]
	\centering
	\includegraphics[width=0.7\linewidth]{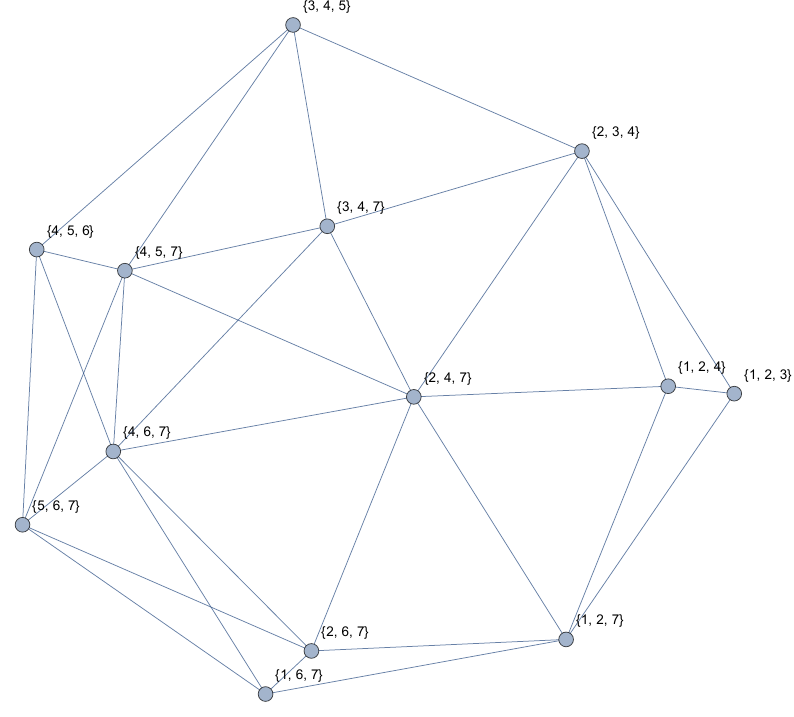}
	\caption{Locations of blades in the arrangement $$\mathcal{C}\cup \{e_{123},e_{234},\ldots, e_{712}\}=\left\{e_{124},e_{247},e_{267},e_{347},e_{457},e_{467},e_{123},e_{234},\ldots, e_{712}\right\}$$ in the hypersimplex $\Delta_{3,7}$.  Thus, in the graph, which is shown embedded in the 1-skeleton of $\Delta_{k,n}$, two vertices are connected by an edge if and only if they differ by a root $e_i-e_j$.  See also \cite{GalashinPostnikovWilliams}.}
	\label{fig:graph37collectionplabicgraph}
\end{figure}

With respect to the given cyclic order $\sigma = ((1,2,3,4,5,6,7))$, the weakly separated collection $\mathcal{C} = \left\{e_{124},e_{247},e_{267},e_{347},e_{457},e_{467}\right\}$ of vertices of $\Delta_{3,7}$ (see Figure \ref{fig:graph37collectionplabicgraph}) determines a collection of decorated ordered set partitions; we shall compute the matroid subdivisions induced by the corresponding blades on the seven boundary copies of $\Delta_{2,6}\simeq  \partial_{(\{j\},1)}(\Delta_{3,7}) \subset \Delta_{3,7}$ and then dualize to get a collection of planar trees.

Denote $\beta_J = ((1,2,\ldots, n))_{e_J}$ for $J\in \binom{\lbrack n\rbrack}{k}$. 

The matroid subdivision corresponding to $\mathcal{C}$ is induced by the superposition of the following blades:
\begin{eqnarray}\label{eqn:collection37}
	\beta_{124}\cap \Delta_{3,7} & = & ((12567_2 34_1)) \cap \Delta_{3,7}\nonumber\\
	\beta_{247}\cap \Delta_{3,7} & = & ((12_1 34_1 567_1)) \cap \Delta_{3,7}\nonumber\\
	\beta_{267}\cap \Delta_{3,7} & = & ((12_1 34567_2)) \cap \Delta_{3,7}\\
	\beta_{347}\cap \Delta_{3,7} & = & ((1234_2 567_1)) \cap \Delta_{3,7}\nonumber\\
	\beta_{457}\cap \Delta_{3,7} & = & ((12345_2 67_1)) \cap \Delta_{3,7}\nonumber\\
	\beta_{467}\cap \Delta_{3,7} & = & ((1234_1 567_2)) \cap \Delta_{3,7}\nonumber.
\end{eqnarray}

We obtain via Lemma \ref{lem: to the boundary hypersimplex} the seven blade arrangements:
\begin{eqnarray*}
	\partial_1(\mathcal{C}) & = & \{e_{24},e_{47},e_{57}\} \mapsto \{((2567_1 34_1)),((234_1 567_1)),((2345_1 67_1))\} \\
	\partial_2(\mathcal{C}) & = & \{e_{14},e_{47},e_{57}\} \mapsto \{((1567_1 34_1)),((134_1 567_1)), ((1345_1 67_1))\} \\
	\partial_3(\mathcal{C}) & = & \{e_{27},e_{47},e_{57}\} \mapsto \{((12_1 4567_1)),((124_1 567_1)),((1245_1 67_1))\} \\
	\partial_4(\mathcal{C}) & = & \{e_{27},e_{37},e_{57}\} \mapsto \{((12_1 3567_1)),((123_1 567_1)),((1235_1 67_1))\}\\
	\partial_5(\mathcal{C}) & = & \{e_{24},e_{27},e_{47}\} \mapsto \{((1267_1 34_1)),((12_1 3467_1)),((1234_1 67_1))\}\\
	\partial_6(\mathcal{C}) & = & \{e_{24},e_{27},e_{47}\} \mapsto \{((1257_1 34_1)),((12_1 3457_1)),((1234_1 57_1))\}   \\
	\partial_7(\mathcal{C}) & = & \{e_{24},e_{26},e_{46}\} \mapsto \{((1256_1 34_1)),((12_1 3456_1)),((1234_1 56_1))\}.
\end{eqnarray*}
Note again that one may use Lemma \ref{lem: to the boundary hypersimplex} to pass directly from the weakly separated collection $$\mathcal{C} = \left\{e_{124},e_{247},e_{267},e_{347},e_{457},e_{467}\right\}$$ of vertices of $\Delta_{3,7}$ to seven weakly separated collections of vertices of $\Delta_{2,6}$, after deleting frozen and redundant vertices.

These can in turn be represented by the arrangement of seven trees in Figure \ref{fig:tree-arrangement}.  The three 2-block set partitions in each line above are determined by the internal edges of the corresponding tree below, as can be seen directly by inspection.

\begin{figure}[h!]
	\centering
	\includegraphics[width=0.8\linewidth]{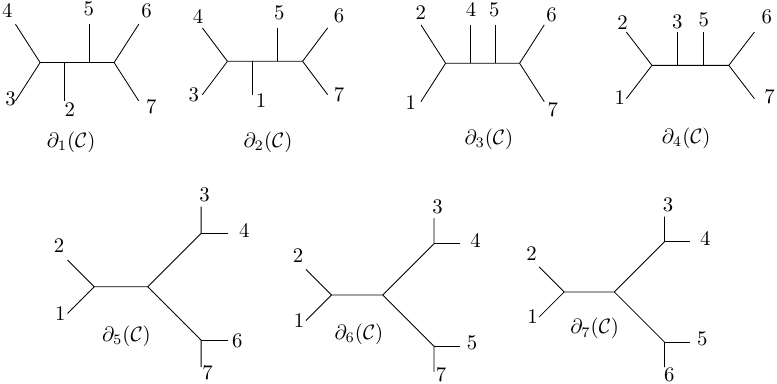}
	\caption{Matroid subdivision induced by the blade arrangement of Equation \eqref{eqn:collection37}, represented as an arrangement of trees dual to the matroid subdivisions induced on the seven boundary copies of $\Delta_{2,6}$.}
	\label{fig:tree-arrangement}
\end{figure}

In Examples \ref{example: D48} and \ref{example: bicolored trees} we summarize some preliminary findings for $\Delta_{4,8}$, derived by enumerating weakly separated collections of vertices.
\begin{example}\label{example: D48}
	Note that for $\Delta_{4,8}$, where there are $\binom{8}{4}-8 = 62$ nonfrozen vertices, maximal weakly separated collections of vertices have $(4-1)(8-4-1)=9$ elements. 
	
	We find 1048 weakly separated pairs.  The new feature at $k=4$ is the 4-split induced by the blade $\beta_{e_{2468}}$. 
	
	One interesting feature is that exactly 24 such pairs contain the vertex $e_{2468}$, while analogously for $(3,6)$ there are 6 weakly separated pairs containing the vertex $e_{246}$.  
	 
Below we give one of the more complex matroidal blade arrangements, derived from the maximal weakly separated collection of vertices 
$$\left\{e_{1248},e_{1268},e_{1468},e_{2348},e_{2468},e_{3458},e_{3468},e_{4568},e_{4678}\right\}.$$
From this we obtain the maximal matroidal blade arrangement (here ``$\equiv $'' means that the two sides agree when intersected with the hypersimplex $\Delta_{4,8}$):
 \begin{eqnarray*}
	\beta_{1248}  & \equiv & ((125678_3 34_1))\\
	\beta_{1268} & \equiv & ((1278_3 3456_1))\\
	\beta_{1468} & \equiv & ((178_1 234_1 56_1))\\
	\beta_{2348} & \equiv & ((1234_3 5678_1))\\
	\beta_{2468} & \equiv & ((12_1 34_1 56_1 78_1))\\
	\beta_{3458} & \equiv & ((12345_3 678_1))\\
	\beta_{3468} & \equiv & ((1234_2 56_1 78_1))\\
	\beta_{4568} & \equiv & ((123456_3 78_1))\\
	\beta_{4678} & \equiv & ((1234_1 5678_3)).
\end{eqnarray*}

\end{example}

\begin{rem}
	Another new feature of $k=4$: while for $\Delta_{3,6}$ the superposition of the blades $((12_1 34_1 56_1))$ and $((12_1 56_1 34_1))$ induce a matroid subdivision, for $\Delta_{4,8}$ the analogous property does not hold: by directly checking the basis exchange relations on the maximal cells, one can verify that the blades 
	$((12_1 34_1 56_12 78_1)$ and $((12_1 78_1 56_1 34_1))$ induce a subdivision that is \textit{not} matroidal.
\end{rem}
We conclude with one of the simplest maximal matroidal blade arrangements: here each blade $((1,2,\ldots, n))_{e_J}$ induces a 2-split of $\Delta_{4,8}$.

\begin{example}\label{example: bicolored trees}
	Let
	$$\mathcal{C} = \left\{e_{2678}, e_{3678},e_{2378},e_{4678},e_{3478},e_{2348}, e_{4578},e_{3458}, e_{4568}\right\}.$$
	This corresponds to the collection of blades
	\begin{eqnarray*}
		&((12_1 345678_3)), ((123_1 45678_3)), ((123_2 45678_2)), ((1234_1 5678_3)),((1234_2 5678_2)), ((1234_3 5678_1)),&\\
		&((12345_2 678_2)),((12345_3 678_1)),((123456_3 78_1)).&
	\end{eqnarray*}
\end{example}

\section{Acknowledgements}
We thank A. Guevara, S. Mizera and J. Tevelev for helpful comments on the final draft, F. Borges and F. Cachazo for many discussions about collections of Feynman diagrams and related topics, and D. Speyer for comments and for pointing out a typo in Corollary \ref{cor:blade splits}.  We also thank the referees for a careful reading.

This research was supported in part by Perimeter Institute for Theoretical Physics. Research at Perimeter Institute is supported by the Government of Canada through the Department of Innovation, Science and Economic Development Canada and by the Province of Ontario through the Ministry of Research, Innovation and Science.

\appendix
\section{Enumerating matroidal blade arrangements}

We enumerated maximal weakly separated collections, (that is, the number of maximal matroidal blade arrangements on $\Delta_{k,n}$) in Mathematica with the help of the FindClique algorithm; the counts are given in the table below, for rows with $n=4,5,\ldots, 12$ and columns with $k=2,3,\ldots, n-2$.

$$\begin{tabular}{c|ccccccccc}
$n\setminus k$ &&&&&&&&\\
\hline 
4 & 2 &  &  &  &  &  &  & & \\
5  & 5 & 5 &  &  &  &  &  &  &\\
6 & 14 & 34 & 14 &  &  &  &  &  & \\
7 & 42 & 259 & 259 & 42 &  &  &  &  & \\
8 & 132 & 2136 & 5470 & 2136 & 132 &  &  & &  \\
9 & 429 & 18600 & 122361 & 122361 & 18600 & 429 & & &  \\
10 & 1430 & 168565 & 2889186 & 7589732 & 2889186 & 168565 & 1430 & & \\
11 & 4862 & 1574298 &71084299 & & &71084299 & 1574298 & 4862 & \\
12 & 16796 & 15051702 & & & & & & 15051702 & 16796
\end{tabular} $$

Likewise, the number of maximal weakly separated collections containing only subsets with two cyclic intervals (that is, the number of maximal matroidal blade arrangements on $\Delta_{k,n}$ where each blade induces a 2-split) is given below.  
$$\begin{tabular}{c|cccccccccccc}
$n\setminus k$ &&&&&&&&\\
\hline 
4 &  2 &  &  &  &  &  &  &  &  \\
5&5 & 5 &  &  &  &  &  &  &  \\
6&14 & 18 & 14 &  &  &  &  &  &  \\
7&42 & 63 & 63 & 42 &  &  &  &  &  \\
8&132 & 216 & 328 & 216 & 132 &  &  &  &  \\
9&429 & 729 & 1683 & 1683 & 729 & 429 &  &  &  \\
10&1430 & 2430 & 8530 & 13020 & 8530 & 2430 & 1430 &  &  \\
11&4862 & 8019 & 42801 & 99825 & 99825 & 42801 & 8019 & 4862 &  \\
12&16796 & 26244 & 212988 & 759204 & 1161972 & 759204 & 212988 & 26244 & 16796 \\
\end{tabular} $$

%
%

\end{document}